\documentclass[reqno]{amsart}
\RequirePackage[l2tabu, orthodox]{nag}

\usepackage{amssymb} 

\usepackage[style=numeric]{biblatex}
\usepackage{caption}
\usepackage{color}
\usepackage{enumitem}
\usepackage{etoolbox}
\usepackage{graphicx}
\usepackage{mathrsfs}
\usepackage{lmodern}
\usepackage{microtype}
\usepackage{stmaryrd}
\SetSymbolFont{stmry}{bold}{U}{stmry}{m}{n}

\usepackage{hyperref} 
\makeatletter
\DeclareFontFamily{OMX}{MnSymbolE}{}
\DeclareSymbolFont{MnLargeSymbols}{OMX}{MnSymbolE}{m}{n}
\SetSymbolFont{MnLargeSymbols}{bold}{OMX}{MnSymbolE}{b}{n}
\DeclareFontShape{OMX}{MnSymbolE}{m}{n}{
    <-6>  MnSymbolE5
   <6-7>  MnSymbolE6
   <7-8>  MnSymbolE7
   <8-9>  MnSymbolE8
   <9-10> MnSymbolE9
  <10-12> MnSymbolE10
  <12->   MnSymbolE12
}{}
\DeclareFontShape{OMX}{MnSymbolE}{b}{n}{
    <-6>  MnSymbolE-Bold5
   <6-7>  MnSymbolE-Bold6
   <7-8>  MnSymbolE-Bold7
   <8-9>  MnSymbolE-Bold8
   <9-10> MnSymbolE-Bold9
  <10-12> MnSymbolE-Bold10
  <12->   MnSymbolE-Bold12
}{}

\let\llangle\@undefined
\let\rrangle\@undefined
\DeclareMathDelimiter{\llangle}{\mathopen}%
                     {MnLargeSymbols}{'164}{MnLargeSymbols}{'164}
\DeclareMathDelimiter{\rrangle}{\mathclose}%
                     {MnLargeSymbols}{'171}{MnLargeSymbols}{'171}
\makeatother

\theoremstyle{plain} 
\newtheorem{theorem}{Theorem}[section]

\newtheorem{proposition}[theorem]{Proposition}
\newtheorem{lemma}[theorem]{Lemma} 

\theoremstyle{definition} 
\newtheorem{definition}[theorem]{Definition}

\theoremstyle{remark} 
\newtheorem{remark}[theorem]{Remark}

\definecolor{shadecolor}{rgb}{1,0.8,0.3}

\newcommand{\IGNORE}[1]{}

\newcommand{\C}{\mathscr{C}}
\newcommand{\E}{\mathcal{E}}
\renewcommand{\L}{\mathscr{L}}
\newcommand{\M}{\mathscr{M}}
\newcommand{\N}{\mathbb{N}}

\newcommand{\Q}{\mathbb{Q}}
\newcommand{\R}{\mathbb{R}}
\renewcommand{\S}{\mathrm{sym}}
\newcommand{\T}{\mathrm{T}}
\newcommand{\W}{\mathscr{W}}

\newcommand{\AF}{{\mathfrak{A}}}
\newcommand{\BL}{\mathrm{BL}}
\newcommand{\BM}{{\boldsymbol{m}}} % need parantheses
\newcommand{\BU}{\boldsymbol{u}}
\newcommand{\BV}{\boldsymbol{v}}
\newcommand{\BW}{\boldsymbol{w}}
\newcommand{\CB}{\C_\mathrm{b}}
\newcommand{\DR}{\dot{\R}}
\newcommand{\FX}{\mathfrak{x}}
\newcommand{\GS}{\geqslant}
\newcommand{\HA}{{\TST\frac{1}{2}}}
\newcommand{\ID}{\mathrm{id}}
\newcommand{\LM}{\boldsymbol{\lambda}}
\newcommand{\LS}{\leqslant}
\newcommand{\MK}{\M_\BL}
\newcommand{\PP}{\mathbb{P}}
\newcommand{\PR}{\boldsymbol{P}}

\newcommand{\SP}{\mathscr{P}}

\newcommand{\ADM}{\mathrm{Adm}}
\newcommand{\ALG}{\mathscr{A}}
\newcommand{\BBE}{\mathbb{E}}
\newcommand{\BFX}{\mathfrak{X}}
\newcommand{\BVS}{\mathrm{BV}}
\newcommand{\DDC}{\boldsymbol{Q}}
\newcommand{\DDP}{\phi}
\newcommand{\DSC}{\boldsymbol{R}}
\newcommand{\DSP}{\chi}
\newcommand{\DST}{\displaystyle}
\newcommand{\EPS}{\varepsilon}
\newcommand{\INT}{\mathcal{U}}
\newcommand{\LEB}{\mathcal{L}}
\newcommand{\LIP}{\mathrm{Lip}}
\newcommand{\LOC}{\mathrm{loc}}
\newcommand{\MAT}[2][]{%
	\ifstrempty{#1}{%
		\mathrm{Mat}_{#2}(\R)}{%
		\mathrm{Mat}_{#2}(\R, #1)}}
\newcommand{\ONE}{\boldsymbol{1}}
\newcommand{\PHI}{\boldsymbol{\phi}}
\newcommand{\RHO}{\varrho}
\newcommand{\SYM}[2][]{%
	\ifstrempty{#1}{%
		\mathrm{Sym}_{#2}(\R)}{%
		\mathrm{Sym}_{#2}(\R, #1)}}
\newcommand{\TST}{\textstyle}
\newcommand{\WAS}{\mathrm{W}}

\newcommand{\FLUX}{\mathbb{U}}
\newcommand{\PREC}{\preccurlyeq}

\newcommand{\SKEW}[1]{{\mathrm{Skew}_{#1}(\R)}}
\newcommand{\WEAK}{\DOTSB\protect\relbar\protect\joinrel\rightharpoonup}

\newcommand{\GAMMA}{\boldsymbol{\gamma}}
\newcommand{\TRACE}{{\mathrm{tr}}}

\DeclareMathOperator{\REL}{\square}

\DeclareMathOperator*{\ESUP}{\mathrm{ess\,sup}}

\numberwithin{equation}{section}

\title 
    [Minimal acceleration for the isentropic Euler equations] 
	{Minimal acceleration for the multi-dimensional isentropic Euler equations}

\author
    {Michael Westdickenberg}

\address
    {Michael Westdickenberg,
     Lehrstuhl f\"{u}r Mathematik (Analysis),
     RWTH Aachen University,
     Templergraben 55,
     D-52062 Aachen, 
     Germany}
\email
    {mwest@instmath.rwth-aachen.de}

\date{\today}

	% Initial value problems for first-order hyperbolic equations

\begin{document}

\begin{abstract} 
On the set of dissipative solutions to the multi-dimensional isentropic Euler
equations, we introduce a quasi-order by comparing the acceleration at all times.
This quasi-order is continuous with respect to a suitable notion of convergence
of dissipative solutions. We establish the existence of minimal elements.
Minimizing the acceleration amounts to selecting dissipative solutions that are
as close to being a weak solution as possible.
\end{abstract}

\maketitle 
% \tableofcontents

%%%%%%%%%%%%%%%%%%%%%%%%%%%%%%%%%%%%%%%%%%%%%%%%%%%%%%%%%%%%%%%%%%%%%%%%%%%%%%%%
%%%%%%%%%%%%%%%%%%%%%%%%%%%%%%%%%%%%%%%%%%%%%%%%%%%%%%%%%%%%%%%%%%%%%%%%%%%%%%%%
%%%%%%%%%%%%%%%%%%%%%%%%%%%%%%%%%%%%%%%%%%%%%%%%%%%%%%%%%%%%%%%%%%%%%%%%%%%%%%%%
%%%%%%%%%%%%%%%%%%%%%%%%%%%%%%%%%%%%%%%%%%%%%%%%%%%%%%%%%%%%%%%%%%%%%%%%%%%%%%%%
%%%%%%%%%%%%%%%%%%%%%%%%%%%%%%%%%%%%%%%%%%%%%%%%%%%%%%%%%%%%%%%%%%% Introduction

\section{Introduction}

The isentropic Euler equations
\begin{equation}
  	\left.\begin{array}{r}
	    \displaystyle
    	\partial_t\RHO
      		+\nabla\cdot(\RHO\BU) = 0
\\[1ex]
    	\displaystyle
    	\partial_t(\RHO\BU)
      		+ \nabla\cdot(\RHO\BU\otimes\BU) 
      		+ \nabla P(\RHO) = 0
  	\end{array}\right\}
  	\quad\text{in $[0,\infty)\times\R^d$}
\label{E:IEE}
\end{equation}
model the evolution of compressible gases. The unknowns $(\RHO,\BU)$ depend on
time $t\in[0,\infty)$ and space $x\in\R^d$. We assume that suitable initial data
is given:
\[
	(\RHO,\BU)(0,\cdot) =: (\bar{\RHO},\bar{\BU}).
\]
We will think of $\RHO$ as a map from $[0,\infty)$ into the space of
non-negative, finite Borel measures, which we denote by $\M_+(\R^d)$. The
quantity $\RHO$ is called the density and it represents the distribution of mass
in time and space. The first equation in \eqref{E:IEE} (the continuity equation)
expresses the local conservation of mass, where
\begin{equation}
    \BU(t,\cdot) \in \L^2\big( \R^d,\RHO(t,\cdot) \big)
    \quad\text{for all $t\in[0,\infty)$}
\label{E:VELOC}
\end{equation}
is the Eulerian velocity field, which takes values in $\R^d$. The second
equation in \eqref{E:IEE} (the momentum equation) expresses the local
conservation of momentum $\BM := \RHO\BU$. Notice that $\BM(t,\cdot)$ is a
finite $\R^d$-valued Borel measure absolutely continuous with respect to
$\RHO(t,\cdot)$ for all $t\in[0,\infty)$, because of \eqref{E:VELOC}. Without
loss of generality, we will asume that $\bar{\RHO}$ has mass one. It then
follows from the continuity equation that for all $t$ we have $\RHO(t,\cdot) \in
\SP(\R^d)$, the space of Borel probability measures.

In order to obtain a closed system \eqref{E:IEE} it is necessary to prescribe an
equation of state, which determines the pressure. For the isentropic Euler
equations, where the thermodynamical entropy is assumed constant in time and
space, the pressure is just a function of the density. We will consider
polytropic gases:

%========== DEFINITION
\begin{definition}[Internal Energy]\label{D:INT} 
Let $U(r) := \kappa r^\gamma$ for all $r\GS 0$, where $\kappa>0$ and $\gamma>1$
are constants. For all $\RHO \in \SP(\R^d)$ we define the internal energy
\[
  	\INT[\RHO] := \begin{cases}
    	\DST \int_{\R^d} U\big( r(x) \big) \,dx
        	& \text{if $\RHO=r\LEB^d$,}
\\
      	+\infty
        	& \text{otherwise.}
    \end{cases}
\]
Here $\LEB^s$ is the $d$-dimensional Lebesgue measure on $\R^d$.
\end{definition}

The constant $\gamma$ is called the adiabatic coefficient. Since we are only
interested in solutions of \eqref{E:IEE} with finite energy, the density
$\RHO(t,\cdot)$ must be absolutely continuous with respect to the Lebesgue
measure for all $t\in[0,\infty)$. Let $r(t,\cdot)$ be its
Lebesgue-Radon-Nikod\'{y}m derivative. Then $p(t,\cdot) = P\big( r(t,\cdot)
\big) \,\LEB^d$ for all $t\in[0,\infty)$, where
\[
    P(r) = U'(r)r-U(r) 
    \quad\text{for $r\GS 0$.}
\]
For simplicity, we will often not distringuish between $\RHO$ and its Lebesgue
density.

Smooth solutions $(\RHO,\BU)$ of \eqref{E:IEE} satisfy the additional
conservation law
\begin{equation}
    \partial_t \Big( \HA\RHO|\BU|^2 + U(\RHO) \Big) + \nabla \cdot 
        \bigg( \Big( \HA\RHO|\BU|^2 + U'(\RHO)\RHO \Big) \BU \bigg) 
            = 0,
\label{E:EE}
\end{equation}
which expresses local conservation of total energy
\[
	E(\RHO,\BU) := \HA \RHO|\BU|^2 + U(\RHO),
\]
which is the sum of kinetic and internal energy. It is well-known, however, that
a generic solution to the isentropic Euler equations will not remain regular,
even for smooth initial data. Instead the solution will develop jump
discontinuities along codimension-one submanifolds in space-time, which are
called shocks. In this case, continuity and momentum equation must be considered
in the sense of distributions, and the energy equation \eqref{E:EE} does not
follow automatically. A physically reasonable relaxation is to assume that no
energy can be created by the fluid: Then the energy equality in \eqref{E:IEE}
must be replaced by the inequality
\begin{equation}
    \partial_t \Big( \HA\RHO|\BU|^2 + U(\RHO) \Big) + \nabla \cdot 
        \bigg( \Big( \HA\RHO|\BU|^2 + U'(\RHO)\RHO \Big) \BU \bigg) 
            \LS 0
\label{E:ED}
\end{equation}
distributionally. Strict inequality in \eqref{E:ED} means that mechanical energy
is transformed into a form of energy not accounted for by the model, such as
heat.

A differential inequality like \eqref{E:ED} contains some information on the
regularity of solutions: The space-time divergence of a certain non-linear
function of $(\RHO,\BU)$ is a non-positive distribution, and therefore a
measure. In the one-dimensional case, one usually requires that weak solutions
of \eqref{E:IEE} satisfy differential inequalities analogous to \eqref{E:ED} for
a large class of non-linear functions of $(\RHO,\BU)$ called entropy-entropy
flux pairs. Such an assumption is called an entropy condition. Utilizing the
method of compensated compactness, it is then possible to establish the global
existence of weak (entropy) solutions of \eqref{E:IEE}; see\cite{Chen2000,
DingChenLuo1987, DingChenLuo1988, ChenLeFloch2000, ChenPerepelitsa2010,
DiPerna1983, LeFlochWestdickenberg2007, LionsPerthameSouganidis1996,
LionsPerthameTadmor1994}.

In several space dimensions the only available entropy-entropy flux pair is the
total energy and the correspondig energy flux. The compensated compactness
technique cannot be applied. One can, however, establish the existence of a
large set of initial data for which weak solutions of \eqref{E:IEE} exist
globally in time, by using non-linear iteration schemes like the ones introduced
by Nash \cite{Nash1954, Nash1956} in the context of isometric embeddings of
Riemannian manifolds. We refer the reader to the ground-breaking work by De
Lellis and Sz\'{e}kelyhidi \cite{DeLellisSzekelyhidi2009,
DeLellisSzekelyhidi2010} and subsequent extensions \cite{Chiodaroli2014,
ChiodaroliKreml2014, ChiodaroliDeLellisKreml2015, Feireisl2014} by various
authors. These results even show that for suitable initial data there exist
\emph{infinitely many} weak solutions of \eqref{E:IEE}, even if one requires
that solutions satisfy an entropy condition in the form \eqref{E:ED}. This is
related to the fact that---in addition to energy dissipation through
shocks---there is an additional dissipation mechanism due to very high
oscillations of the velocity field, which is reminiscent of anomalous
dissipation in turbulence. Moreover, there is a precise threshold of H\"{o}lder
regularity $1/3$ between energy conserving and energy dissipating regimes. For
incompressible flows, this has been conjectured based on physical considerations
by Onsager \cite{Onsager1949}. A mathematical proof of this conjecture has been
provided in a series of recent articles; see \cite{Buckmaster2015,
BuckmasterDeLellisSzekelyhidi2016, Isett2018} and references therein. For
related results for the compressible Euler equations see
\cite{FeireislGwiazdaSwierczewskaGwiazdaWiedemann2017}. The \emph{Cauchy
problem} for \eqref{E:IEE} in several space dimensions has not been solved yet:
In order to apply the above methods for \emph{given} initial data, it is
necessary to allow a small increase in energy initially, which violates
\eqref{E:ED}. 

We do not know whether weak solutions of \eqref{E:IEE} always exist. Clearly
there is no uniqueness. It is conceivable that there is initial data for which
no well-defined weak solution can be found, as numerical experiments in
\cite{LimYuGlimmLiSharp2008, LimIwerksGlimmSharp2010, FjordholmMishraTadmor2016}
involving particular situations such as Kelvin-Helmholtz instabilities have
suggested. It has been argued that the apparent non-convergence of numerical
solutions under successive mesh refinement necessitates the use of solution
concepts that are weaker than weak solutions, like measure-valued or dissipative
solutions; see also \cite{FjordholmMishraTadmor2016,
FeireislLukacovaMedvidovaMizerova2020}.

\subsection*{Dissipative solutions}

In order to construct approximate solutions of \eqref{E:IEE}, we will use the
variational time discretization introduced in
\cite{CavallettiSedjroWestdickenberg2019}, which is inspired by minimizing
movement schemes. For given timestep $\tau>0$ it generates approximate solutions
at discrete times $t_\tau^k := k\tau$ with $k\in \N_0$ by solving an
optimization problem for each step, minimizing the sum of a suitable work
functional plus the internal energy. This can be interpreted as maximizing the
entropy production; see \cite{CavallettiSedjroWestdickenberg2019}. The work
functional quantifies how much mass points, which at the beginning of the time
step have a well-defined location and velocity, deviate from the free transport
path, i.e., it measures their acceleration. It is a second-order replacement of
the Wasserstein distance, which plays a central role in the interpretation of
certain parabolic equations as gradient flows on the space of probability
measures; see \cite{AmbrosioGigliSavare2008}. 

Having determined approximate solutions of \eqref{E:IEE} at discrete times
$t_\tau^k$, we then interpolate in time to obtain a curve, i.e., a map $t
\mapsto (\RHO_\tau, \BM_\tau)(t,\cdot)$ into the state space, for all $t \in
[0,\infty)$. There are different ways to do this. One possibility is to
interpolate piecewise constantly in time so that the approximate solution jumps
at $t_\tau^k$. Another one, which was not considered in
\cite{CavallettiSedjroWestdickenberg2019}, utilizes a piecewise linear
interpolation of the transport map and the velocity. We will need both
interpolations. The second one provides a tighter control of the total energy;
see Remark~\ref{R:CHANGES}.

We then consider the limit $\tau \rightarrow 0$ and obtain, extracting
subsequences as necessary, a limit density/momentum $(\RHO,\BM)$, which is the
candidate for a solution of \eqref{E:IEE}. As in
\cite{CavallettiSedjroWestdickenberg2019}, we use Young measures to capture the
weak limits of non-linear functions of $(\RHO_\tau, \BM_\tau)$ as $\tau\to 0$.
Since we consider different time interpolations, we obtain two Young measures
$(\epsilon, \nu)$. Young measures are elements of the dual space
\[
  	\BBE^* := \L^\infty_w\big( [0,\infty), \M(\dot\R^d\times\BFX) \big),
\]
which is the space of weakly measurable, essentially bounded maps from
$[0,\infty)$ into the space of finite Borel measures on $\dot\R^d\times\BFX$.
Here $\dot\R^d$ is the one-point compacti\-fication of the physical space, and
$\BFX$ is a suitable compactification of the state space; see
Section~\ref{SS:YM}. In the limit $\tau\to 0$, we then have the convergence
\[
	\BM_\tau\otimes\BM_\tau / \RHO_\tau 
		\WEAK \lbrack \BM\otimes\BM / \RHO \rbrack
	\quad\text{and}\quad
	P(\RHO_\tau) \WEAK \llbracket P(\RHO) \rrbracket
\]
in a suitable topology, where the brackets $\lbrack \cdot \rbrack$ and
$\llbracket \cdot \rrbracket$ denote the pairing with Young measures $\nu$ and
$\epsilon$, respectively. We refer the reader to Section~\ref{SS:YM} for
details. Then $(\RHO,\BM)$ satisfy the continuity equation and a modified
momentum equation
\begin{equation}
	\partial_t \BM + \nabla \cdot \FLUX = 0
	\quad\text{with}\quad
	\FLUX := \lbrack \BM\otimes\BM / \RHO \rbrack
		+ \llbracket P(\RHO) \rrbracket \ONE.
\label{E:YMR}
\end{equation}
We will call such pairs of Young measures \emph{dissipative solutions},
borrowing the term from \cite{BreitFeireislHofmanova2020a}. In that paper,
dissipative solutions are defined as tuples of $(\RHO,\BM)$ and \textbf{defect
measures} $\DSC, \DDP$ that satisfy the continuity equation and
\begin{equation}
    \partial_t \BM
        + \nabla \cdot \bigg( \frac{\BM\otimes\BM}{\RHO} + P(\RHO)\ONE \bigg)
        + \fbox{$\nabla \cdot (\DSC + \DDP\ONE)$} = 0
\label{E:MOM2}
\end{equation}
distributionally. Here $\DSC$ and $\DDP$ are measures taking values in the
symmetric, positive semidefinite matrices and the non-negative numbers,
respectively. Moreover
\begin{equation}
    \frac{d}{dt} \int_{\R^d} \Bigg(
        \HA \RHO|\BU|^2 + U(\RHO) 
            +\fbox{$\DST \HA \TRACE(\DSC) + \frac{1}{\gamma-1} \DDP$}
                \, \Bigg) (t,dx) \LS 0
\label{E:EIG}
\end{equation}
distributionally, which provides a bound on the sizes of $\DSC$ and $\DDP$. The
dissipative solutions of \cite{BreitFeireislHofmanova2020a} have the
\emph{weak-strong uniqueness property}, meaning that dissipative solutions
coincide with strong solutions of \eqref{E:IEE} as long as the latter exist. We
will show in Proposition~\ref{P:EXISTENCE} that our momentum flux has a similar
decomposition
\begin{equation}
  \FLUX = \bigg( \frac{\BM\otimes\BM}{\RHO} + P(\RHO)\ONE \bigg)
    + \Big( \DSC + \DDP\ONE \Big),
\label{E:ACCI}
\end{equation}
with defect measures $\DSC$ and $\DDP$ that are positive semidefinite as above.
This follows from Jensen's inequality and another Young measure representation.
We prefer to think of dissipative solutions as pairs of Young measures
$(\epsilon, \nu)$ because these come with a natural topology of weak*
convergence that is very useful for studying the behavior of non-linear
functions of density/momentum; see Section~\ref{S:MA}. For a weak-strong
uniqueness result for solutions of the isentropic Euler equations with possible
occurrence of vacuum we refer the reader to \cite{GhoshalJanaWiedemann2021}.

\subsection*{Time regularity}

As is well-known, the underlying local geometry plays a central role in defining
(generalized) gradient flows. In the classical case of a gradient flow of some
functional $E$ defined on a Riemannian manifold $(\mathcal{M},g)$, for example,
the Riemannian metric $g$ is used to associate to the differential $dE(x)$ at a
point $x \in \mathcal{M}$, which is a cotangent vector, the gradient $\nabla
E(x)$, which is tangent. Then
\begin{equation}
	\dot\gamma(t) = -\nabla E\big( \gamma(t) \big)
	\quad\text{for $t \in [0,\infty)$}
\label{E:GFE}
\end{equation}
makes sense as an equality of tangent vectors. For generalized gradient flows on
the space of probability measures, the Wasserstein distance induces a suitable
notion of tangent vectors as the set of square-integrable velocity fields
$\BV$ such that
\begin{enumerate}[label=(\arabic*)]
\item the continuity equation $\partial_t \RHO = - \nabla\cdot(\RHO\BV)$ is
satisfied distributionally;
\item the $\L^2(\RHO_t)$-norm of $\BV_t$ is minimal, for a.e.\ $t\in [0,\infty)$.
\end{enumerate}
Here $t \mapsto \RHO_t$ with $t\in [0,\infty)$ is a sufficiently smooth curve of
probability measures. We use the subscript $t$ to indicate the evaluation at
time $t$. Notice that the continuity equation does not uniquely determine
$\BV_t$. It is always possible to add another square-integrable vector field
$\BW_t$ with $\nabla\cdot(\RHO_t\BW_t) = 0$ without losing the equality. The
minimization of the $\L^2(\RHO_t)$-norm in (2) amounts to defining tangent
vectors to the space of probability measures as (limits of) \emph{gradient}
vector fields. 

If the functional $E$ is convex, but non-smooth, then the gradient flow equation
\eqref{E:GFE} must be relaxed to $-\dot\gamma(t)$ being contained in the
\emph{subdifferential} of $E$ at the point $\gamma(t)$, which is a set-valued
map. The evolution typically selects the elements of minimal length (norm) in
this subdifferential; see \cite{Brezis1973, AmbrosioGigliSavare2008}. In this
sense, gradients flows combine the idea of steepest descent (proceed in the
direction of the subdifferential, so that the functional decreases as fast as
possible) and the idea of slowest evolution (among the vectors in the
subdifferential select the one with minimal norm).

When studying a variational approach for the isentropic Euler equation
\eqref{E:IEE}, the question arises what is the ``correct'' metric structure for
the space of momenta, which are vector-valued Borel measures. We propose to use
the bounded Lipschitz norm (also known as Monge-Kantorovich norm). It is defined
by testing against bounded, Lipschitz continuous functions with norm less than
or equal to one; see Definition~\ref{D:BLNOM}. The topological dual of the space
of bounded, Lipschitz continuous functions is very large, containing besides
bounded measures also certain distributions of first order. However, the set of
measures with \emph{total variation uniformly bounded} (e.g., because of an
energy bound) turns out to be \emph{closed} in this dual space. This is
well-known in the case of probability measures where the bounded Lipschitz norm
metrizes the weak* convergence of measure; see \cite{Dudley1966}. A similar
thing happens here. We refer the reader to \cites{Pachl1979, Pachl2013,
ChitescuMikulescuNitaIoana2016, HilleSzarekWormZiemlanska2021} for further
information.

The use of the bounded Lipschitz norm is very well adapted to the momentum
equation $\partial_t \BM = -\nabla\cdot \FLUX$, which provides an expression for
the \emph{time derivative} of the momentum. To measure its size we test against
bounded, Lipschitz continuous functions, integrate by parts, then take the
supremum over all test functions with norm not bigger than $1$. Since $\FLUX$
takes values in the symmetric, positive semi\-definite matrices, this turns out
to be very simple: We just need to integrate the trace $\TRACE(\FLUX)$ to obtain
an upper bound. Since the total energy is bounded for all times, it follows that
the momentum $t \mapsto \BM_t$ is \emph{Lipschitz continuous} with respect to
the bounded Lipschitz norm. The Lipschitz continuity of $t \mapsto \RHO_t$ with
respect to the Wasserstein distance follows as in
\cite{AmbrosioGigliSavare2008}. We refer the reader to Section~\ref{SS:TR} for
details.

\subsection*{Minimal acceleration}

In addition to the bounded Lipschitz norm, there is another quantity that is
significant for our discussion. Recall that if $(\mathscr{S}, d)$ is a complete
metric space, then a curve $t \mapsto v(t)$ into $\mathscr{S}$ is called
\emph{absolutely continuous} if there exists an integrable function $m \in
\L^1_\LOC([0,\infty))$ such that
\begin{equation}
	d\big( v(t), v(s) \big) \LS \int_s^t m(r) \,dr
	\quad\text{for all $0\LS s\LS t <\infty$.}
\label{E:MEDER}
\end{equation}
One can show that for any absolutely continuous curve $v$ in $\mathscr{S}$ the
limit
\[
	|v'|(t) := \lim_{s\to t} \frac{d\big(v(s),v(t)\big)}{|s-t|}
	\quad\text{exists for a.e.\ $t\in [0,\infty)$,}
\]
and $|v'|$ is the smallest function one can use in place of $m$ on the
right-hand side of \eqref{E:MEDER}; see Theorem~1.1.2 in
\cite{AmbrosioGigliSavare2008}. This $|v'|$ is called the \emph{metric
derivative} of $v$.

We apply this construction to the set of momentum fields $\BM$, which are
vector-valued Borel measures. Because of the energy bound, the total variations
of these momentum measures are bounded uniformly in time. Moreover, it suffices
to consider measures with prescribed average: by using the constant as a test
function in the momentum equation in \eqref{E:IEE}, one finds that the spatial
integral of $\BM$ (the total momentum) is preserved in time. On momentum curves
$t\mapsto \BM_t$, we define
\begin{equation}
	d\big( \BM_s, \BM_t \big) := \sup\bigg\{ 
		\int_{\R^d} \zeta(x) \cdot \big( \BM_s(dx)-\BM_t(dx) \big) \colon
			\|\zeta\|_{\LIP(\R^d)} \LS 1 \bigg\}
\label{E:DULIP}
\end{equation}
for $0\LS s\LS t$. Note that the difference between \eqref{E:DULIP} and the
distance induced by the bounded Lipschitz norm is that in \eqref{E:DULIP} we do
not require the $\sup$-norm of the test function to be bounded by one, only its
Lipschitz semi-norm. This is sufficient because $\BM_s$ and $\BM_t$ have the
same total momentum; testing their difference against the constant function
returns zero. We refer the reader to \cite{ChitescuMikulescuNitaIoana2016} for a
related discussion. We will show in Lemma~\ref{L:MOMS} that, with momentum flux
\[
	\FLUX(t,dx) = \lbrack \BM\otimes\BM/\RHO 
		\rbrack(t,dx) + \llbracket P(\RHO) \rrbracket(t,dx) \ONE,
\]
the metric derivative of $t\mapsto\BM_t$ induced by the distance \eqref{E:DULIP}
satisfies
\[
	|\BM'|(t) = \int_{\DR^d} \TRACE\big( \FLUX(t,dx) \big)
	\quad\text{for a.e.\ $t\in[0,\infty)$.}
\]
In the following, we will be interested in dissipative solutions of
\eqref{E:IEE} that minimize the metric derivative $|\BM'|$. We also say we try
to \emph{minimize the acceleration}. 

Given the fluid state $(\RHO_t,\BM_t)$ at some time $t$, there could be many
admissible momentum fluxes $\FLUX$ that are consistent with the notion of
dissipative solutions. This set, however, has the shape of a convex cone. Indeed
we have \eqref{E:ACCI} with non-negative defect measures $\DSC$ and $\DDP$ as
discussed above. It is therefore not possible to make the momentum flux $\FLUX$
arbitrarily small. It is only possible to bring the momentum flux as close as
possible to the vertex $\BM\otimes\BM/\RHO + P(\RHO) \ONE$ of the convex cone.
If we could reach this vertex for a.e.\ $(t,x)$, then the dissipative solution
would actually be a \emph{weak} solution. We refer to reader to
\cite{FeireislHofmanova2020} for a similar discussion. 

\medskip

Therefore our goal here is to select from the set of dissipative solutions such
solutions that \textbf{minimize the acceleration for almost all times}. 

\medskip

This problem is structurally very similar to multi-objective optimization, which
is an optimization problem that involves \emph{multiple} objective functions
$f_1(x), \ldots, f_k(x)$ defined on some set $X$. In this paper, we are
interested in minimizing the metric derivative (acceleration) for almost all
times, so our objective functions are indexed by $t$, and minimization is done
over the set of dissipative solutions.

Since in multi-objective optimization it is typically not possible to find an $x
\in X$ that \emph{minimizes all $f_i(x)$ simultaneously}, different strategies
have been devised to deal with such problems. One strategy consists in
prioritizing the objective functions (or combinations thereof) and minimizing
iteratively, meaning that one selects the minimizers of $f_{i+1}$ not from all
of $X$, but \emph{from the set of minimizers of $f_i$}, which has been selected
in the previous step, for $i=1\ldots k-1$. This procedure can also work for
countably many objective functions: Since the sets of minimizers are
\emph{nested} it is sufficient to prove that each such set is non-empty and
compact. Then the intersection of all sets is non-empty, because of the Cantor
intersection theorem. Typically the procedure is sensitive to the
\emph{ordering} of the objective functions.

An alternative approach to multi-objective optimization is to replace the goal
of finding a minimizer by the goal of finding a \emph{minimal element}. That
means, to find an element $m \in X$ that cannot be improved in any of the
objective functions without making things worse for another objective. Such
elements are called \emph{Pareto optimal}. More precisely, we say that $x_1 \in
X$ Pareto dominates $x_2 \in X$ if
\begin{equation}
\begin{aligned}
  f_i(x_1) \LS f_i(x_2) 
    & \quad\text{for all $i \in \{1, \ldots, k\}$,}
\\
  f_j(x_1) < f_j(x_2)
    & \quad\text{for at least one $j \in \{1, \ldots, k\}$.}
\end{aligned}
\label{E:PARETO}
\end{equation}
We say that $m \in X$ is Pareto optimal if there exists no $x \in X$ that Pareto
dominates $m$. This can be rephrased as saying that $m$ is a \emph{minimal
element} with respect to the quasi-order defined by \eqref{E:PARETO}. Note that
there may be pairs $x_1, x_2 \in X$ that cannot be compared in the sense that
neither $x_1$ Pareto dominates $x_2$, nor vice versa. Moreover, \eqref{E:PARETO}
typically does not define a partial order because $f_i(x_1) = f_i(x_2)$ for all
$i$ may not imply that $x_1 = x_2$. We refer the reader to Section~\ref{S:QOS}
for more.

\medskip

For suitable initial data $(\bar{\RHO}, \bar{\BM})$, let $S$ be the set of
dissipative solutions with this initial data; see Definition~\ref{D:DISSI}. For
any $(\epsilon, \nu) \in S$ and all $t\in[0,\infty)$ we define
\begin{equation}
	a(t|\epsilon, \nu) := 
	\int_{\DR^d} \Big(	\TRACE\big( \lbrack \BM\otimes\BM/\RHO 
			\rbrack(t,dx) \big)	+ d \llbracket P(\RHO) \rrbracket(t,dx) \Big).
\label{E:INTRACE}
\end{equation}
The map $t \mapsto a(t|\epsilon, \nu)$ is in $\L^\infty([0,\infty))$ because of
Definition~\ref{D:DISSI}~\eqref{L:EPSILON} and \eqref{L:NU}.

%========== DEFINITION
\begin{definition}\label{D:QO} 
We define a quasi-order $\PREC$ on $S$: We say that $(\tilde\epsilon,
\tilde\nu) \PREC (\epsilon, \nu)$ if
\[
	a(t|\tilde\epsilon, \tilde\nu) \LS a(t|\epsilon, \nu)
	\quad\text{for a.e.\ $t\in [0,\infty)$.}
\]
\end{definition}

It is straightforward to check that $\PREC$ is reflexive and transitive (see
Section~\ref{S:QOS}); it may not be antisymmetric. Then we have the following
result:

%========== THEOREM
\begin{theorem}
\label{T:SELECTION}
For any $(\bar{\RHO}, \bar{\BM})$ as in \eqref{E:DATA}, let $S$ be the set of
dissipative solutions  of the isentropic Euler equations \eqref{E:IEE} with
initial data $(\bar{\RHO}, \bar{\BM})$. Then there exists a minimal element
$(\check\epsilon, \check\nu) \in S$ with respect to the quasi-order $\PREC$ of
Definition~\ref{D:QO}: For any $(\epsilon, \nu) \in S$, if $(\epsilon, \nu)$
and $(\check\epsilon, \check\nu)$ are comparable at all, then 
\[
	a(t|\check\epsilon, \check\nu) \LS a(t|\epsilon, \nu)
	\quad\text{for a.e.\ $t\in[0,\infty)$.}
\]
In this sense, the solution $(\check\epsilon, \check\nu)$ minimizes the
acceleration.
\end{theorem}

The proof of Theorem~\ref{T:SELECTION} will be given in Section~\ref{S:MA}.

%========== REMARK
\begin{remark}
As the proof of Theorem~\ref{T:SELECTION} shows, the minimal elements are
extracted from subsets of $S$ that are \emph{totally ordered} by the quasi-order
$\PREC$ and maximal with respect to inclusion. Such subsets do exist because of
the Hausdorff maximality principle. It would be interesting to know whether
there are dissipative solutions that minimize the acceleration \emph{over the
whole set $S$}, not only over maximal totally ordered subsets. This is an open
problem (and may well be false).
\end{remark}

%========== REMARK
\begin{remark}
Because of \eqref{E:ACCI}, minimizing the acceleration amounts to making the
defect measures $\DSC$ and $\DDP$ as small as possible. Heuristically, one might
expect that the mininization of acceleration can counteract the occurence of
highly oscillatory velocity fields, which can be the source of anomalous
dissipation. By minimizing the defect measures one brings the dissipative
solution as close to being a weak solution as possible; see
\cite{FeireislHofmanova2020} for a related discussion. It would be interesting
to check whether one can reach $\DSC = 0$ and $\DDP = 0$ in subsets of
$\R^{d+1}$ so that the dissipative solution becomes a weak solution of
\eqref{E:IEE} there.
% Working with measure-valued/dissipative solutions could make it easier to
% construct candidates for comparison/perturbation, maybe using results from
% \cite{DePhilippisRindler2016}.
This will be considered elsewhere.

To make a connection between our approach and the construction of infinitely
many weak solutions of the isentropic Euler equation \eqref{E:IEE} using convex
integration, we sketch the argument in \cite{Markfelder2020}. The starting point
is a triple $(\RHO, \BM, \FLUX_0)$ with
\begin{equation}
	\partial_t \RHO + \nabla \cdot \BM = 0,
	\quad
	\partial_t \BM + \nabla \cdot \FLUX_0 = 0 
\label{E:SUBSOL}
\end{equation}
in some space-time domain $Q$, where $(\RHO, \BM)$ are smooth density/momentum
and $\FLUX_0$ is a smooth field taking values in the \emph{trace-free},
symmetric matrices. It is assumed that the following inequality holds pointwise
everywhere:
\begin{equation}
	d \lambda_\text{max} \Bigg( \bigg( \frac{\BM\otimes\BM}{\RHO} 
		+ P(\RHO)\ONE \bigg) - \FLUX_0 \Bigg) < c,
\label{E:LMAX}
\end{equation}
for some positive constant $c$. Here $\lambda_\text{max}(A)$ is the largest
eigenvalue of the symmetric $(d\times d)$-matrix $A$. The triple $(\RHO, \BM,
\FLUX_0)$ is called a \emph{subsolution}. 

Since $c$ is constant we can replace $\FLUX_0$ in \eqref{E:SUBSOL} by the
momentum flux
\begin{equation}
	\FLUX := \frac{c}{d} \ONE + \FLUX_0,
\label{E:UPLUS}
\end{equation}
and then \eqref{E:LMAX} is similar to the condition $\BM\otimes\BM/\RHO +
P(\RHO) \ONE \LS \FLUX$ in the sense of symmetric matrices; see \eqref{E:ACCI}.
In order to construct a weak solution of \eqref{E:IEE} the goal is now to
iteratively modify $(\RHO, \BM, \FLUX_0)$, while preserving \eqref{E:SUBSOL}, so
that equality is achieved a.e.\ in \eqref{E:LMAX} in the limit. This can be done
by the repeated superposition of oscillatory waves and a Baire category
argument, as pioneered in \cite{DeLellisSzekelyhidi2009,
DeLellisSzekelyhidi2010}.

Interestingly, it is sufficient to consider the trace
\begin{align}
	\Big( |\BM|^2/\RHO + d P(\RHO) \Big)
		& = \TRACE\bigg( \frac{\BM\otimes\BM}{\RHO} + P(\RHO) \ONE - \FLUX_0 \bigg)
\nonumber\\
		& \LS d \lambda_\text{max} \Bigg( \bigg( \frac{\BM\otimes\BM}{\RHO} 
		+ P(\RHO)\ONE \bigg) - \FLUX_0 \Bigg),
\label{E:LMAXTRACE}
\end{align}
which is smaller than $c$ everywhere because of \eqref{E:LMAX}. Indeed the
iteration in \cite{Markfelder2020} is set up to achieve equality in the limit
for the upper bound
\begin{equation}
	\int_Q \Big( |\BM|^2/\RHO + d P(\RHO) \Big) \LS c |Q|,
\label{E:LEFR}
\end{equation}
with $|Q|$ the Lebesgue measure of $Q$. The integral is a lower bound of the
acceleration. But instead of making \eqref{E:INTRACE} small, the goal in
\cite{Markfelder2020} is to make the left-hand side of \eqref{E:LEFR} bigger. If
$|\BM|^2/\RHO + d P(\RHO) = c$ a.e., then equality holds in \eqref{E:LMAXTRACE}
and
\[
	\frac{\BM\otimes\BM}{\RHO} + P(\RHO) \ONE - \FLUX_0 
		= \frac{1}{d} \TRACE\bigg( \frac{\BM\otimes\BM}{\RHO} + P(\RHO) \ONE 
			- \FLUX_0 \bigg) \ONE 
		= \frac{c}{d} \ONE
	\quad\text{a.e.;}
\]
see Lemma~0.2.1 in \cite{Markfelder2020}. Therefore $\FLUX = \BM\otimes\BM/\RHO
+ P(\RHO) \ONE$ a.e.\ (recall \eqref{E:UPLUS}), which we can insert into
\eqref{E:SUBSOL} in place of $\FLUX_0$ to obtain a weak solution of
\eqref{E:IEE}.

% It would be interesting to further explore this connection. Whereas for the
% convex integration approach, the existence of subsolutions needs to be
% established first, which subsequently can be improved to a weak solution of
% \eqref{E:IEE}, for our approach the existence of minimal elements is given (as a
% consequence of the Hausdorff maximality principle) but it needs to be shown that
% such a minimal element corresponds to a weak solution. Note that in both
% approaches the axiom of choice plays an important role, for example in the Baire
% category theorem.
\end{remark}

%========== REMARK
\begin{remark}
We do not know whether acceleration minimizing, dissipative solutions are
unique. In order to overcome the non-uniqueness issue for hyperbolic
conservation laws, it has been proposed in \cite{Dafermos1973} to look for
solutions that decrease the total entropy (here: the total energy) as quickly as
possible. The right-hand side of \eqref{E:INTRACE}: 
\begin{align*}
	& \int_{\DR^d} \Big(
		\TRACE\big( \lbrack \BM\otimes\BM/\RHO \rbrack(t,dx) \big)
			+ d \llbracket P(\RHO) \rrbracket(t,dx) \Big)
\\
	&\qquad
		= 	\int_{\DR^d} \Big( 2 \big\lbrack \HA|\BM|^2/\RHO \big\rbrack(t,dx)
				+ d (\gamma-1) \llbracket U(\RHO) \rrbracket(t,dx) \Big)
\end{align*}
is a total energy, up to numerical factors. In this sense, our selection
procedure searches for dissipative solutions with minimal energy. Kinetic and
internal energies are somewhat independent since they involve different Young
measures.

Alternatively, we can apply the selection procedure above directly to
\begin{equation}
	f(t|\epsilon, \nu) := \int_{\DR^d} \llbracket \HA|\BM|^2/\RHO + U(\RHO)
		\rrbracket(t,dx)
\label{E:FENEGR}
\end{equation}
and use $f$ instead of $a$ to define the quasi-order $\PREC$. The integrand in
\eqref{E:FENEGR} is bounded below by $\HA|\BM|^2/\RHO + U(\RHO)$, with
density/momentum $(\RHO, \BM)$; see
Definition~\ref{D:DISSI}~\eqref{L:DECOMPOSITION}. We obtain a dissipative
solution that minimizes the total energy at a.e.\ $t\in[0,\infty)$ in the sense
of Theorem~\ref{T:SELECTION}, i.e., along maximal totally ordered subsets. Again
it would be interesting to find energy minimizers of the whole set $S$.

For a related discussion, we refer the reader to
\cite{BreitFeireislHofmanova2020b}. In that paper, the authors consider a
countable family of objective functionals, defined as the Laplace transforms of
the energy profiles $t \mapsto f(t| \epsilon, \nu)$ for a suitable collection of
Laplace parameters, and then iteratively minimize these functionals, as
explained above in the context of multi-objective optimization. They get an
energy minimizing dissipative solution in that sense. Analogously, instead of
constructing minimal elements using a quasi-order, as we are doing here, we
could apply the iterative minimization of \cite{BreitFeireislHofmanova2020b} to
find acceleration minimizing dissipative solutions in the sense outlined above.

The main focus of \cite{BreitFeireislHofmanova2020b}, however, is not on energy
minimization but the construction of a semigroup, i.e., of a solution operator
that to any admissible initial data selects precisely one dissipative solution
such that the following holds: If one follows a dissipative solution up to some
time $t$ and then takes the density/momentum at that time as new initial data,
then the solution operator will select for this new initial data precisely the
original dissipative solution shifted back in time by $t$. The issue here is the
restart-ability of dissipative solutions. This may be relevant to the
construction of dissipative solutions with minimal acceleration profiles in the
sense of the quasi-order in Definition~\ref{D:QO} and will be considered
elsewhere.
\end{remark}

%%% IGNORE

%%%%%%%%%%%%%%%%%%%%%%%%%%%%%%%%%%%%%%%%%%%%%%%%%%%%%%%%%%%%%%%%%%%%%%%%%%%%%%%%
%%%%%%%%%%%%%%%%%%%%%%%%%%%%%%%%%%%%%%%%%%%%%%%%%%%%%%%%%%%%%%%%%%%%%%%%%%%%%%%%
%%%%%%%%%%%%%%%%%%%%%%%%%%%%%%%%%%%%%%%%%%%%%%%%%%%%%%%%%%%%%%%%%%%%%%%%%%%%%%%%
%%%%%%%%%%%%%%%%%%%%%%%%%%%%%%%%%%%%%%%%%%%%%%%%%%%%%%%%%%%%%%%%%%%%%%%%%%%%%%%%
%%%%%%%%%%%%%%%%%%%%%%%%%%%%%%%%%%%%%%%%%%%%%%%%%%%%%%%%%%% Quasi-Ordered Spaces

\section{Quasi-Ordered Spaces}
\label{S:QOS}

A \emph{quasi-order} on a set $X$ is a binary relation $R$ with the properties
\begin{enumerate}[label=(\roman*)]
\item Reflexivity: for all $x \in X$ we have $(x,x) \in R$.
\item Transitivity: for all $x,y, z \in X$ we have
\[
	\text{$(x,y) \in R$ and $(y,z) \in R$}
	\quad\Longrightarrow\quad
	(x,z) \in R.
\]
\end{enumerate}
Note that we do not assume
\begin{enumerate}[label=(\roman*), resume]
\item Antisymmetry: for all $x,y \in X$ we have
\[
	\text{$(x,y) \in R$ and $(y,x) \in R$}
	\quad\Longrightarrow\quad
	x=y,
\]
\end{enumerate}
which would make the quasi-order $R$ into a \emph{partial order}. To simplify
the notation, we will usually write $x \PREC y$ instead of $(x,y) \in R$, with
$x,y \in X$. 

For any $x \in X$ we define the set of predecessors as
\[
	P(x) := \{ y \in X \colon y \PREC x \}.
\]
%

%========== THEOREM
\begin{theorem}\label{T:WARD}
Suppose that $X$ is a non-empty compact set with a quasi-order $R$ such that
$P(x)$ is closed for every $x\in X$. Then $X$ has a minimal element, i.e., an
element $m \in X$ such that, if $y \in X$ and $m$ can be compared at all, then $m
\PREC y$.
\end{theorem}

%========== PROOF
\begin{proof}
This was proved in \cite{Wallace1945}. We include the argument for the reader's
convenience. We first observe that the set $A := \{ P(x) \colon x\in X \}$ is
partially ordered by inclusion. By the Hausdorff maximal principle (which is
equivalent to the axiom of choice), there exists a \emph{maximal totally ordered
subset} $T \subset A$, i.e., it holds
\begin{enumerate}
\item For all $P(x_1), P(x_2) \in T$ we have $P(x_1) \subseteq P(x_2)$ or
$P(x_2) \subseteq P(x_1)$.
\item $T$ is not properly contained in another totally ordered set.
\end{enumerate}
Then the intersection $\bigcap_{P(x) \in T} P(x)$ is non-empty. Indeed if this
were false, then
\[
	X = X \setminus \bigcap_{P(x) \in T} P(x) 
		= \bigcup_{P(x) \in T} \big( X \setminus P(x) \big),
\]
by De Morgan's laws. Since all $P(x)$ are closed, their complements are open,
and so 
\[
	\text{$\{X\setminus P(x)\}_{P(x)\in T}$ is an open covering of $X$.}
\]
Since $X$ is compact, there exists a finite subcovering, i.e., there exist
finitely many $x_k \in X$, $k=1\ldots K$, for which $P(x_k) \in T$ and the
following holds:
\[
	X = \bigcup_{k=1}^K \big( X \setminus P(x_k) \big)
		= X \setminus \bigcap_{k=1}^K P(x_k).
\]
Thus $\bigcap_{k=1}^K P(x_k)$ is empty. But every \emph{finite} subset of a
non-empty totally ordered set has a lower bound; this can be proved using
induction on the cardinality of the subset. This means there exists an index $l
\in \{1 \ldots K\}$ such that
\[
	P(x_l) \subseteq P(x_k)
	\quad\text{for all $k=1\ldots K$,}
\]
and thus $P(x_l) = \bigcap_{k=1}^K P(x_k)$. It follows that $P(x_l)$ must be
empty. But this is a contradiction because $x_l \in P(x_l)$, by reflexivity of
the quasi-order.

Therefore there exists an element $m \in \bigcap_{P(x)\in T} P(x)$. 

We claim that $m$ is a minimal element of $X$. Since $m \in P(x)$ for all $P(x)
\in T$, it follows that $m \preccurlyeq x$ for such $x$. Consider therefore $x
\in X$ with $P(x) \not\in T$. Suppose that $x \preccurlyeq m$. Since $P(x)
\not\in T$ there must exist $P(y) \in T$ to which $P(x)$ cannot be compared,
i.e., it holds neither $P(y) \subseteq P(x)$ nor $P(x) \subseteq P(y)$. Indeed
if $P(x)$ could be compared to every $P(y) \in T$, then $P(x) \in T$, by
maximality of $T$. But $P(y) \in T$ implies that $m \preccurlyeq y$, as we have
just shown. By transitivity, our assumption $x \PREC m$ then implies that $x
\preccurlyeq y$. Using transitivity again, we observe that for every $z \PREC x$
we also have $z \PREC y$, and thus $P(x) \subseteq P(y)$. But this is a
contradiction since $P(x)$ and $P(y)$ cannot be compared, by choice of $y$.
Therefore the case $x \preccurlyeq m$ cannot occur, which means that either $x$
and $m$ cannot be compared at all, or $m \PREC x$.
\end{proof}

%%%%%%%%%%%%%%%%%%%%%%%%%%%%%%%%%%%%%%%%%%%%%%%%%%%%%%%%%%%%%%%%%%%%%%%%%%%%%%%%
%%%%%%%%%%%%%%%%%%%%%%%%%%%%%%%%%%%%%%%%%%%%%%%%%%%%%%%%%%%%%%%%%%%%%%%%%%%%%%%%
%%%%%%%%%%%%%%%%%%%%%%%%%%%%%%%%%%%%%%%%%%%%%%%%%%%%%%%%%%%%%%%%%%%%%%%%%%%%%%%%
%%%%%%%%%%%%%%%%%%%%%%%%%%%%%%%%%%%%%%%%%%%%%%%%%%%%%%%%%%%%%%%%%%%%%%%%%%%%%%%%
%%%%%%%%%%%%%%%%%%%%%%%%%%%%%%%%%%%%%%%%%%%%%%%%%%%%%%%%%% Dissipative Solutions

\section{Dissipative Solutions}

In this section, we introduce the class of dissipative solutions of
\eqref{E:IEE} from which we will select the ones with minimal acceleration. Our
construction is more detailed than the one in \cite{FeireislGhoshalJana2020}
since it involves Young measures to describe the defect measures $\DSC, \DSP$ in
\eqref{E:MOM2}; see also \cite{FeireislHofmanova2020}. Let us first introduce
some notation. 

Let $\MAT{d}$ be the space of real $(d\times d)$-matrices and
\[
	\MAT[\square]{d} := \big\{ A \in \MAT{d} \colon 
		\text{$v\cdot (Av) \REL 0$ for all $v\in\R^d$} \big\}
\]
where $\REL$ stands for either $\GS$ or $>$. The analogous spaces of symmetric
matrices are denoted by $\SYM{d}$ and $\SYM[\REL]{d}$. We denote by $\SKEW{d}$
the space of real skew symmetric $(d\times d)$-matrices. The Frobenius inner
product is defined as
\[
	A : B := \TRACE(A^\T B)
	\quad\text{for $A, B \in \MAT{d}$.}
\]
Let $\|\cdot\|$ be the operator norm induced by the Euclidean norm $|\cdot|$ on
$\R^d$.
%
% \[
% 	\|A\| := \sup_{\xi\in\R^d, |\xi|=1} |A\xi|
% 	\quad\text{for all $A \in d}$.}
% \]
%
% For any $A \in d}$ we define as
% %
% \[
% 	A^\S := (A+A^\T)/2
% 	\quad\text{and}\quad
% 	A^\A := (A-A^\T)/2
% \]
% %
% its symmetric and antisymmetric part, respectively. 

We denote by $\DR^d$ the one-point compactification of $\R^d$: We adjoin to
$\R^d$ a point $\infty$ and define, with $h(x) := 1/(1+|x|)$ for all $x\in\R^d$,
a distance
\[
    d(x,y) := \begin{cases}
        \min\{ |x-y|, h(x)+h(y) \}
            & \text{if $x,y \in \R^d$,}
\\
        h(x) 
            & \text{if $x\in\R^d$ and $y=\infty$,}
\\
        0
            & \text{if $x,y=\infty$;}
   \end{cases}
\]
see \cite{Mandelkern1989}. Then $|x|\rightarrow \infty$ is equivalent to
$d(x,\infty) \rightarrow 0$. 

We denote by $\C_*(\R^d; V)$ the space of continuous functions 
\[
	g \colon \R^d \longrightarrow V
	\quad\text{for which}\quad
	\text{$\lim_{|x|\rightarrow\infty} g(x) \in V$ exists,}
\]
equipped with the $\sup$-norm. Here $V$ is some Banach space. Then
\[
	\C_*(\R^d; V) = V + \C_0(\R^d; V),
\]
with $\C_0(\R^d; V)$ the closure of the space of compactly supported continuous
$V$-valued functions with respect to the the $\sup$-norm. Functions in
$\C_*(\R^d; V)$ can be identified with elements in $\C(\DR^d; V)$: To
$g\in\C_*(\R^d; V)$ we associate $\dot g \in \C(\dot\R^d; V)$ as
\[
    \dot g(x) := \begin{cases}
        g(x) 
            & \text{if $x\in\R^d$,}
\\
        \lim_{|x|\rightarrow\infty} g(x)
            & \text{if $x=\infty$.}
    \end{cases}
\]
For simplicity of notation, we will not distinguish between $g$ and $\dot{g}$.

We now define the space of test functions
\begin{equation}
	\AF := \big\{ u\in\C^1(\R^d;\R^D) \colon
    \text{$\nabla u \in \C_*\big( \R^d;\MAT{D\times d} \big)$} \big\},
\label{E:AF}
\end{equation}
with $\MAT{D\times d}$ the space of $(D\times d)$-matrices. We will not
explicitly indicate the dimension $D$ when it is clear from the context.
Functions in $\AF$ grow at most linearly at infinity. In particular, the space
$\AF$ contains all linear maps 
\[
	u(x):=Ax
	\quad\text{for all $x\in \R^d$, with $A\in\MAT{d}$,}
\]
and so test functions do not have to have compact support.

Let $\C^1_c([0,\infty)) \otimes \AF$ be the space of tensor products
\begin{equation}
	\eta\otimes\zeta(t,x) := \eta(t)\zeta(x)
	\quad\text{with $\eta \in \C^1_c\big( [0,\infty) \big)$ and 
		$\zeta \in \AF$.}
\label{E:TENS}
\end{equation}
We will assume that partial differential equations hold in duality with
$\C^1_c([0,\infty)) \otimes \AF$, which means testing against functions of the
form \eqref{E:TENS}; see Section~\ref{SS:GE} for details. For all $T>0$, the
tensor product $\C([0,T])\otimes\AF$ is dense in $\C([0,T]; \AF)$ with respect
to the $\sup$-norm because $\AF$ is a locally convex topological vector space. 

%%%%%%%%%%%%%%%%%%%%%%%%%%%%%%%%%%%%%%%%%%%%%%%%%%%%%%%%%%%%%%%%%%%%%%%%%%%%%%%%
%%%%%%%%%%%%%%%%%%%%%%%%%%%%%%%%%%%%%%%%%%%%%%%%%%%%%%%%%%%%%%%%%%%%%%%%%%%%%%%%
%%%%%%%%%%%%%%%%%%%%%%%%%%%%%%%%%%%%%%%%%%%%%%%%%%%%%%%%%%%%%%%% A Priori Bounds

\subsection{A Priori Bounds}
\label{SS:APB}

We now collect various natural a priori-bounds for solutions of \eqref{E:IEE}.
To simplify notation, we will use the subscript $t$ to indicate the value at
time $t$, as in $\RHO_t := \RHO(t,\cdot)$. The continuity equation in
\eqref{E:IEE} implies that if $\bar{\RHO}$ has total mass one, then $\RHO_t \in
\SP(\R^d)$ for all $t\in [0,\infty)$. Similarly, the entropy condition
\eqref{E:ED} implies that the total energy should be non-increasing in time.
Assuming that
\begin{equation}
	\bar{E} := \int_{\R^d} \HA \bar{\RHO} |\bar{\BU}|^2 + U[\bar{\RHO}]
		< +\infty
\label{E:ENIN}
\end{equation}
we will therefore be interested in solutions of \eqref{E:IEE} with finite
energy, so that
\begin{equation}
	\int_{\R^d} \HA \RHO_t |\BU_t|^2 + U[\RHO_t] \LS \bar{E}
	\quad\text{for all $t \in [0,\infty)$.}
\label{E:ENRGY}
\end{equation}
As a consequence of the Cauchy-Schwarz inequality, the momentum $\BM_t := \RHO_t
\BU_t$ is an $\R^d$-valued Borel measure whose total variation is bounded by
$\bar{E}$, uniformly in $t$. Moreover, Definition~\ref{D:INT} and
\eqref{E:ENRGY} provide higher integrability for $\RHO_t$. We have
\[
	\RHO \in \L^\infty\big( [0,\infty); \L^\gamma(\R^d) \big),
	\quad
	\BM \in \L^\infty\big( [0,\infty); \L^\frac{2\gamma}{\gamma+1}(\R^d) \big),
\]
because of H\"{o}lder inequality. In particular, both $\RHO_t$ and $\BM_t$ are
absolutely continuous with respect to $\LEB^d$. We will always make assumption
\eqref{E:ENIN} in the following.

We will also require that the initial density $\bar{\RHO}$ has finite second
moment:
\begin{equation}
	\bar{M} := \bigg( \int_{\R^d} |x|^2 \,\bar{\RHO}(dx) \bigg)^{1/2}
		< +\infty.
\label{E:MOIN}
\end{equation}
Multiplying the continuity equation by $\HA|x|^2$ and integrating by parts, we
find
\[
	\frac{d}{dt} \bigg( \int_{\R^d} |x|^2 \,\RHO_t(dx) \bigg)^{1/2}
		\LS \bigg( \int_{\R^d} |\BU_t(x)|^2 \,\RHO_t(dx) \bigg)^{1/2}
\]
for all $t$, formally. We will therefore be interested in solutions of
\eqref{E:IEE}, for which
\begin{equation}
	\bigg( \int_{\R^d} |x|^2 \,\RHO_t(dx) \bigg)^{1/2}
		\LS \bar{M} + t (2\bar{E})^{1/2} =: M(t)
	\quad\text{for all $t\in [0,\infty)$.}
\label{E:MOMM}
\end{equation}
This in turn implies that the momentum $\BM_t$ has finite first moment for all
times, which follows again from Cauchy-Scharz inequality with \eqref{E:ENRGY}.

%========== REMARK
\begin{remark}
The products of $(\RHO_t, \BM_t)$ with $\zeta \in \AF$ are integrable in space
since these measures have finite first moments. Moreover, the spatial derivative
$\nabla\zeta$ is bounded, therefore the integrals involving fluxes of
\eqref{E:IEE} are well-defined as well.
\end{remark}

%%%%%%%%%%%%%%%%%%%%%%%%%%%%%%%%%%%%%%%%%%%%%%%%%%%%%%%%%%%%%%%%%%%%%%%%%%%%%%%%
%%%%%%%%%%%%%%%%%%%%%%%%%%%%%%%%%%%%%%%%%%%%%%%%%%%%%%%%%%%%%%%%%%%%%%%%%%%%%%%%
%%%%%%%%%%%%%%%%%%%%%%%%%%%%%%%%%%%%%%%%%%%%%%%%%%%%%%%%%%%%%%%% Time Regularity

\subsection{Time Regularity}\label{SS:TR}

Because of the a priori-bounds from Section~\ref{SS:APB}, we think of solutions
of the isentropic Euler equations \eqref{E:IEE} as curves $t \mapsto (\RHO_t,
\BM_t)$ taking values in a convex set of vector measures whose total variations
are bounded uniformly in time. In order to quantify the time regularity of these
curves we must choose an appropriate metric structure on the spaces of densities
and momenta.

%========== DEFINITION
\begin{definition}[$p$-Wasserstein Distance]
For any $\RHO^1, \RHO^2 \in \SP(\R^D)$ let
\label{D:WAS}
\[ 
    \ADM(\RHO^1,\RHO^2) := \big\{ \GAMMA\in\SP(\R^{2D}) \colon
        \text{$\PP^k\#\GAMMA = \RHO^k$ with $k=1..2$} \big\}
\]
be the space of admissible transport plans connecting $\RHO^1$ and
$\RHO^2$, where
\[
    \PP^k(x^1, x^2) := x^k
    \quad\text{for all $(x^1,x^2) \in \R^{2D} = (\R^D)^2$}
\]
and $k=1..2$, and $\#$ denotes the push-forward of measures. For any
$1\LS p<\infty$ the $p$-Wasserstein distance $\WAS_p(\RHO^1,\RHO^2)$
between $\RHO^1$, $\RHO^2$ is defined by
\begin{equation}
    \WAS_p(\RHO^1,\RHO^2)^p := \inf_{\GAMMA \in \ADM(\RHO^1,\RHO^2)}
        \left\{ \int_{\R^{2D}} |x^1-x^2|^p \,\GAMMA(dx^1,dx^2) \right\}.
\label{E:WASDE}
\end{equation}
\end{definition}

The $p$-Wasserstein distance between the two measures $\RHO^1$ and $\RHO^2$ is
the minimal cost it takes to transport $\RHO^1$ into $\RHO^2$ if the cost of
moving a unit mass in $\R^D$ from $x$ to $y$ is defined as the $p$th power of
the Euclidean distance. The elements of $\ADM(\RHO^1, \RHO^2)$ are called
\emph{transport plans}; they all transport $\RHO^1$ to $\RHO^2$. The $\inf$ in
\eqref{E:WASDE} is attained for a suitable $\GAMMA \in \ADM(\RHO^1, \RHO^2)$,
called an \emph{optimal transport plan}. For $p=2$ the support of an optimal
transport plan is contained in the graph of a cyclically monotone map (i.e, in
the subdifferential of a proper, lower semicontinuous, convex function).

%========== DEFINITION
\begin{definition}
We denote by $\SP_2(\R^d)$ the space of Borel probability measures with finite
second moment, endowed with the $2$-Wasserstein distance; see
Definition~\ref{D:WAS}. For a function $t \mapsto \RHO_t \in \SP_2(\R^d)$,
$t\in[0,\infty)$, we denote by
\[
	\|\RHO\|_{\LIP([0,\infty); \SP_2(\R^d))} 
  		:= \sup_{\substack{t_1, t_2 \in [0,\infty) \\ t_1\neq t_2}} 
    		\frac{\WAS_2(\RHO_{t_1}, \RHO_{t_2})}{|t_2-t_1|}
\]
its Lipschitz seminorm, with $\WAS_2$ the Wasserstein distance; see
\eqref{E:WASDE}.
\end{definition}

For any $t\in [0,\infty)$, the momentum $\BM_t = \RHO_t \BU_t$ is an element of
the convex set
\begin{equation}
	\M_t := \bigg\{ \BM \in \M(\R^d; \R^d) \colon 
		\int_{\R^d} (1+|x|) \,|\BM(dx)|
			\LS \big( 1+M(t) \big) (2\bar{E})^{1/2} \bigg\};
\label{E:MT}
\end{equation}
see Section~\ref{SS:APB}. The measures in $\M_t$ have bounded total variation
and are \emph{uniformly tight}, which means that for all $\EPS>0$ there exists a
compact set $K \subset \R^d$ with
\[
	\int_{\R^d\setminus K} |\BM(dx)| < \EPS
	\quad\text{for all $\BM \in \M_t$.}
\]
Indeed since $\BM \in \M_t(\R^d)$ has finite first moment, we can estimate
\begin{equation}
	\int_{\R^d\setminus \bar{B}_R(0)} |\BM(dx)| 
		\LS R^{-1} \int_{\R^d} |x| \,|\BM(dx)|
	\quad\text{for all $R>0$.}
\label{E:HERE}
\end{equation}
The integral on the right-hand side of \eqref{E:HERE} is bounded by
\eqref{E:MT}. On $\M_t$ the topology of narrow convergence of measures, defined
in terms of testing against bounded continuous functions, coincides with the
topology induced by the bounded Lipschitz norm (also called Dudley or
Monge-Kantorovich norm), defined as follows:

%========== DEFINITION
\begin{definition}\label{D:BLNOM}
We denote by $\LIP(\R^d; \R^D)$ the vector space of Lipschitz continuous maps
$\zeta \colon \R^d \longrightarrow \R^D$. The Lipschitz constant of $\zeta \in
\LIP(\R^d; \R^D)$ is
\[
    \|\zeta\|_{\LIP(\R^d)} 
        := \sup_{x_1 \neq x_2} \frac{|\zeta(x_1)-\zeta(x_2)|}{|x_1-x_2|}.
\]
We denote by $\BL(\R^d;\R^D)$ the subspace of bounded functions in $\LIP(\R^d;
\R^D)$. It is a Banach space when equipped with the norm
\begin{equation}
	\|\zeta\|_{\BL(\R^d)} 
		:= \|\zeta\|_{\L^\infty(\R^d)} + \|\zeta\|_{\LIP(\R^d)}.
\label{E:BKNORM}
\end{equation}
Let $\BL_1(\R^d; \R^D)$ be the space of all $\zeta \in \BL(\R^d; \R^D)$ with
$\|\zeta\|_{\BL(\R^d)}\LS 1$.

We denote by $\MK(\R^d; \R^D)$ the space of finite $\R^D$-valued Borel measures
$\BM$ with finite first moment, equipped with the bounded Lipschitz norm
\begin{equation}
    \|\BM\|_{\MK(\R^d)} := \sup\left\{ \int_{\R^d} \zeta(x) \cdot \BM(dx) 
        \colon \zeta \in \BL_1(\R^d; \R^D) \right\}.
\label{E:KANT} 
\end{equation}
The bounded Lipschitz norm is bounded above by the total variation. The integral
in \eqref{E:KANT} is well-defined because $\BM$ has finite first moment, by
assumption.
\end{definition}

We refer the reader to \cites{ChitescuMikulescuNitaIoana2016,
HilleSzarekWormZiemlanska2021} for a related discussion.

%%%%%%%%%%%%%%%%%%%%%%%%%%%%%%%%%%%%%%%%%%%%%%%%%%%%%%%%%%%%%%%%%%%%%%%%%%%%%%%%
%%%%%%%%%%%%%%%%%%%%%%%%%%%%%%%%%%%%%%%%%%%%%%%%%%%%%%%%%%%%%%%%%%%%%%%%%%%%%%%%
%%%%%%%%%%%%%%%%%%%%%%%%%%%%%%%%%%%%%%%%%%%%%%%%%%%%%%%%%%%%%%%%% Young Measures

% The level sets of the function $h(\RHO,\BM)$ defined below describe the
% boundaries of convex sets in the $\RHO$-$\BM$-space. They always contain the
% point $(0,0)$. Considering all possible non-negative values for $h$ we obtain a
% foliation of $X$ that approaches in particular the vacuum $\{0\}\times\R^d$.
% Topologically, the compactification below is like the inner tube of a ball that
% nestles to the inside of a ball when inflated.

\subsection{Young Measures}
\label{SS:YM}

We will use parameterized measures (Young measures) to describe solutions of
\eqref{E:IEE}. The state space for density/momentum $(\RHO, \BM)$ is
\[
	X \cup \{(0,0)\}, 
    \quad\text{where}\quad 
    X := \Big( (0,\infty) \times \R^d \Big).
\]
Our goal is to define a suitable compactification of the state space.
Equivalently, we must specify the set of continuous and bounded functions on $X$
that are needed to represent the non-linearities in the isentropic Euler
equations \eqref{E:IEE}. In slight abuse of notation, we will use the symbols
$(\RHO, \BM)$ for elements in $X$. Let
\[
	h(\RHO, \BM) 
    	:= \RHO + \bigg( \frac{|\BM|^2}{2\RHO} + U(\RHO) \bigg);
\]
see Definition~\ref{D:INT}. We then introduce the set
\[
  	\W(X) := \left\{ 
    	\begin{aligned}
    		& \Phi = \varphi + \Bigg( 
        		c_\RHO \cdot \begin{pmatrix} \RHO \\ \BM \end{pmatrix}
        		+ c_K : \frac{\BM\otimes\BM}{\RHO}
        		+ c_U P(\RHO) 
      				\Bigg) / h \colon
\\
      		& \hspace{4em}
				\varphi \in \C_0(X),
        		c_\RHO \in \R^{d+1},
        		c_K \in \SYM{d},
        		c_U \in \R
    \end{aligned} \right\}.
\]
One can check that the functions in $\W(X)$ are continuous and bounded.
Moreover, being a finite-dimensional augmentation of the vector space $\C_0(X)$,
which is known to be separable, the set $\W(X)$ is a complete and
\emph{separable} vector space with respect to uniform convergence. Then there
exists a compact, metrizable Hausdorff space $\BFX$  and an embedding $e \colon
X \longrightarrow \BFX$ such that $e(X)$ is dense in $\BFX$. We will call $\BFX$
a \emph{compactification} of $X$. If $\ALG$ denotes the smallest closed
subalgebra in $\CB(X)$ that contains $\W(X)$, then for all maps $\Phi \in \ALG$,
the composition $\Phi \circ e^{-1} \colon X \longrightarrow \R$ has a continuous
(hence bounded) extension to all of $\BFX$. For simplicity of notation, we will
identify $X$ with its image $e(X)$, and functions $\Phi \in \ALG$ with their
extensions in $\C(\BFX)$. We refer the reader to Sections~6.4/5 in
\cite{CavallettiSedjroWestdickenberg2019} for additional details.

With $\DR^d$ the one-point compactification of $\R^d$, we now define 
\begin{equation}
  	\BBE := \L^1\big( [0,\infty),\C(\DR^d\times\BFX) \big)
\label{E:BBE}
\end{equation}
as the space of  measurable maps $\PHI \colon [0,\infty) \longrightarrow
\C(\dot\R^d\times\BFX)$ with finite norm:
\[
  	\|\PHI\|_\BBE 
    	:= \int_0^\infty \|\PHI(t,\cdot)\|_{\C(\dot\R^d\times\BFX)} \,dt
    		< \infty.
\]
Here $\PHI$ is measurable if it is the pointwise limit of a sequence of simple
functions. As usual, we identify functions that differ only on a Lebesgue null
set.

Since $\BFX$ is compact and metrizable it is separable. One can then show that
$\BBE$ is a separable Banach space. Its topological dual is given by
\begin{equation}
  	\BBE^* := \L^\infty_w\big( [0,\infty), \M(\dot\R^d\times\BFX) \big),
\label{E:ESTAR}
\end{equation}
the space of functions $\nu \colon [0,\infty) \longrightarrow
\M(\dot\R^d\times\BFX)$ such that
\begin{gather}
  	\text{$t \mapsto \int_{\dot\R^d\times\BFX} \phi(x, \FX) \,\nu_t(dx, d\FX)$ 
    	measurable for all $\phi \in \C(\dot\R^d\times\BFX)$, and}
\nonumber\\
  	\|\nu\|_{\BBE^*} 
	    := \ESUP_{t\in[0,\infty)} \|\nu_t\|_{\M(\dot\R^d\times\BFX)} 
			< \infty.
\label{E:ESTARBD}
\end{gather}
We used the notation $\nu_t := \nu(t,\cdot)$ and $\FX := (\RHO, \BM) \in \BFX$.
Again we identify functions that coincide almost everywhere. The duality is
induced by the pairing
\[
  	\langle \nu,\PHI\rangle := \int_0^\infty \int_{\dot\R^d\times\BFX} 
		\PHI(t, x, \FX) \,\nu_t(dx, d\FX) \,dt
	\quad\text{for $\PHI \in \BBE$ and $\nu \in \BBE^*$.}
\]
Bounded closed balls in $\BBE^*$ endowed with the weak* topology are metrizable
and (sequentially) compact, by Banach-Alaoglu theorem. We write
\begin{equation}
  \lbrack f(\RHO,\BM) \rbrack_\nu(t,dx) 
    := \int_{\BFX} \frac{f(\FX)}{h(\FX)} \,\nu_t(dx, d\FX)
  \quad\text{for a.e.\ $t\in[0,\infty)$}
\label{E:YPAIR}
\end{equation}
and for all functions $f \colon X \longrightarrow \R$ with $f/h \in \ALG$. We
emphasize that the pairing \eqref{E:YPAIR}, being an integration of $f$ with
respect to $\nu$, is \emph{linear in $f$}. 
% We write the subscript $\nu$ to indicate that the pairing uses the Young measure
% $\nu$.

%========== REMARK
\begin{remark}\label{R:VAC}
The compactfication $\BFX$ adds points not only for the limit of \emph{large}
$(\RHO, \BM)$ but also for the case when $(\RHO,\BM)$ approaches \emph{vacuum}.
Indeed we have that
\begin{equation}
	\lim_ {\RHO\to 0} \frac{1}{h(\RHO,\BM)} \left.\begin{pmatrix}
		\RHO \\ \BM 
	\end{pmatrix}\right|_{\BM=\RHO\BU}
		= \frac{1}{1+|\BU|^2} \begin{pmatrix}
			1 \\ \BU 
		\end{pmatrix},
\label{E:VAC}
\end{equation}
for any $\BU \in \R^d$ fixed. Notice that the right-hand side of \eqref{E:VAC}
vanishes as $|\BU|\to\infty$. The points in $\BFX$ corresponding to vacuum
therefore have the topology of $\R^d$. This can also be seen by using
$(\RHO,\BM) \mapsto (\RHO,\BU := \BM/\RHO)$ to map the (non vacuum)
density-momentum space $X$ to the density-velocity space $Y := [0,\infty) \times
\R^d$, where vacuum is $\{0\}\times\R^d$. Let $V$ denote the set of vacuum
points of the compactification $\BFX$ and $X_V := ((0,\infty) \times \R^d) \cup
V$. The topology of $X_V$ is the one inherited from $\BFX$.
\end{remark}

%%%%%%%%%%%%%%%%%%%%%%%%%%%%%%%%%%%%%%%%%%%%%%%%%%%%%%%%%%%%%%%%%%%%%%%%%%%%%%%%
%%%%%%%%%%%%%%%%%%%%%%%%%%%%%%%%%%%%%%%%%%%%%%%%%%%%%%%%%%%%%%%%%%%%%%%%%%%%%%%%
%%%%%%%%%%%%%%%%%%%%%%%%%%%%%%%%%%%%%%%%%%%%%%%%%%%%%%%%%%%%%%% Global Existence

\subsection{Gobal Existence}
\label{SS:GE}

% It seems that the argument in Theorem~5.2 in \cite{FeireislGhoshalJana2020} is
% not correct. In order to estimate the defect measure terms in (5.5) in that
% paper in terms of the total energy, it is necessary that $\nabla\boldsymbol{U}$
% and $\text{div}\boldsymbol{U}$ are bounded in $\L^\infty$. This does not follow
% from the assumption of Besov regularity stated there. So the correct statement
% is that dissipative solutions coincide with strong solutions if the velocity is
% Lipschitz continuous.

We can now introduce dissipative solutions of \eqref{E:IEE}.

%========== DEFINITION
\begin{definition}[Dissipative Solutions] 
\label{D:DISSI}
Suppose that initial data
\begin{equation}
    \text{$\bar{\RHO} \in \SP_2(\R^d)$ with $\INT[\bar{\RHO}] < +\infty$,}
    \quad
    \bar{\BU} \in \L^2(\R^d,\bar{\RHO}),
	\quad
	\bar{\BM} := \bar{\RHO} \bar{\BU}
\label{E:DATA}
\end{equation}
is given. Let initial energy/moment $\bar{E}, \bar{M}$ be defined by
\eqref{E:ENIN} and \eqref{E:MOIN}.

A \emph{dissipative solution} of the isentropic Euler equations \eqref{E:IEE} is
a pair
\[
	\epsilon, \nu \in \L^\infty_w\big( [0,\infty);
		\M_+(\dot\R^d\times\BFX) \big)
\]
such that there exist 
\begin{equation}
\begin{aligned}
	\RHO & \in \L^\infty\big( [0,\infty); \L^\gamma(\R^d) \big),
\\
	\BM & \in \L^\infty\big( [0,\infty); \L^p(\R^d) \big),
	\quad p := \frac{2\gamma}{\gamma+1},
\end{aligned}
\label{E:SPACES}
\end{equation}
with the properties (1)--(9) listed below. We denote by $\llbracket \cdot
\rrbracket$ and $\lbrack \cdot \rbrack$ the pairing with the Young measures
$\epsilon$ and $\nu$, respectively; see \eqref{E:YPAIR} for the definition.

\medskip

\textbf{A priori bounds}
\begin{enumerate}
\item \label{L:BOUNDS}
The moment bound \eqref{E:MOMM} holds and $\BM_t \in \M_t$ for all $t \in
[0,\infty)$; see \eqref{E:MT}.
\item \label{L:EPSILON}
The map $t \mapsto E(t)$ is non-increasing and bounded by $\bar{E}$, with 
\begin{equation}
	E(t) := \int_{\DR^d} \big\llbracket \HA|\BM|^2/\RHO
		+ U(\RHO) \big\rrbracket(t,dx)
	\quad\text{for all $t\in [0,\infty)$.}
\label{E:TOTEN}
\end{equation}
\item \label{L:NU}
We have $N(t) \LS E(t)$ a.e., where
\begin{equation}
	N(t) := \int_{\DR^d} \big\lbrack \HA|\BM|^2/\RHO 
		+ U(\RHO) \big\rbrack(t,dx)
	\quad\text{for all $t\in [0,\infty)$.}
\label{E:NEN}
\end{equation}
\end{enumerate}

\medskip

\textbf{Time Regularity}
\begin{enumerate}[resume]
\item \label{L:LIPSCHITZ}
There exists a constant $L$ depending only on $d, \gamma$ such that 
\begin{align}
	\|\RHO\|_{\LIP([0,\infty);\SP_2(\R^d))} & \LS (2\bar{E})^{1/2},
\label{E:LIPWAS}\\[1ex]
	\|\BM\|_{\LIP([0,\infty); \MK(\R^d))} & \LS L\bar{E}.
\nonumber
\end{align}
\end{enumerate}

\medskip

\textbf{Eulerian Velocity}
\begin{enumerate}[resume]
\item \label{L:VELOCITY}
We have $\BM =: \RHO\BU$ with 
\[
    \BU_t \in \L^2(\R^d, \RHO_t)
    \quad\text{for all $t \in [0,\infty)$.}
\]
\end{enumerate}

\medskip

\textbf{Young measures}
\begin{enumerate}[resume]
\item \label{L:COMPATIBILITY}
The function $\RHO, \BM$ are compatible with the Young measures $\epsilon, \nu$:
\begin{equation}
\left.
\begin{gathered}
	\RHO_t(dx) 
		= \llbracket \RHO \rrbracket(t,dx) 
		= \lbrack \RHO \rbrack(t,dx)
\\[1ex]
	\BM_t(dx) 
		= \llbracket \BM \rrbracket(t,dx)
		= \lbrack \BM \rbrack(t,dx)
\end{gathered}
\right\}
\quad\text{for a.e.\ $t\in[0,\infty)$}.
\label{E:COMPAT}
\end{equation}
\item \label{L:DECOMPOSITION}
There exist functions
\begin{gather*}
    \DDC, \DSC \in \L_w^\infty\big( [0,\infty); \M(\dot{\R}^d; \SYM[\GS]{d}) \big),
\\[1ex]
    \DDP, \DSP \in \L_w^\infty\big( [0,\infty); \M_+(\dot{\R}^d) \big),
\end{gather*}
such that the following decomposition holds for a.e.\ $t\in[0,\infty)$:
\begin{equation}
\begin{aligned}
	\llbracket \BM\otimes\BM/\RHO \rrbracket(t,dx)
		& = r_t(x) \BU_t(x)\otimes\BU_t(x) \,dx + \DDC(t,dx),
\\[1ex]
	\llbracket P(\RHO) \rrbracket(t,dx)
		& = P\big( r_t(x) \big) \,dx + \DDP(t,dx),
\end{aligned}
\label{E:MOFL}
\end{equation}
where $\RHO_t =: r_t \LEB^d$. Formulas \eqref{E:MOFL} hold analogously with
$\lbrack \cdot \rbrack$ and $\DSC, \DSP$. 
\end{enumerate}

\medskip

\textbf{Hyperbolic conservation laws}
\begin{enumerate}[resume]
\item \label{L:DATA}
The initial data is attained:
\[
    \RHO(0,\cdot) = \bar{\RHO},
    \quad
    \BM(0,\cdot) = \bar{\BM}.
\]
\item \label{L:CONSLAWS}
The conservation laws are satisfied:
\begin{equation}
    \left.\begin{array}{r}
        \partial_t\RHO
            +\nabla\cdot\BM = 0 % ( \RHO\BU ) = 0
\\[0.5em]
        \partial_t \BM %(\RHO\BU)
			+ \nabla \cdot \lbrack \BM\otimes\BM/\RHO \rbrack %\RHO\BU\otimes\BU \rbrack
        	+ \nabla \llbracket P(\RHO) \rrbracket = 0
    \end{array}\right\}
    \quad\text{in $\Big( \C^1_c\big( [0,\infty) \big)\otimes \AF \Big)^*$.}
\label{E:DISSI}
\end{equation}
\end{enumerate}
\end{definition}

%========== REMARK
\begin{remark}\label{R:ENEBO}
Because of Definition~\ref{D:DISSI}~\eqref{L:DECOMPOSITION}/\eqref{L:CONSLAWS},
the pair $(\RHO,\BM = \RHO\BU)$ satisfies the modified momentum equation
\eqref{E:MOM2}. We can rewrite the total energy as
\begin{align}
	& \int_{\DR^d} \llbracket \HA|\BM|^2/\RHO + U(\RHO) \rrbracket(t,dx)
\label{E:REPRE}\\
	& \qquad
		= \int_{\R^d} \Big( \HA r_t|\BU_t|^2 + U(r_t) \Big) \,dx 
			+ \int_{\DR^d} \bigg( \HA\TRACE\big( \DDC(t,dx) \big)
				+ \frac{1}{\gamma-1} \DDP(t,dx) \bigg),
\nonumber
\end{align}
where $\RHO_t = r_t\LEB^d$ and $t\in[0,\infty)$. Then
Definition~\ref{D:DISSI}~\eqref{L:EPSILON} implies the energy inequality
\eqref{E:EIG} with $\DDC$ in place of $\DSC$, from which \eqref{E:ENRGY}
follows. Equality \eqref{E:REPRE} also holds with $(\llbracket\cdot\rrbracket,
\DDC, \DDP)$ replaced by $(\lbrack\cdot\rbrack, \DSC, \DSP)$, because of
\eqref{E:COMPAT}. We find that
\[
	\int_{\DR^d} \bigg( \HA\TRACE\big( \DDC(t,dx) \big)
			+ \frac{1}{\gamma-1} \DDP(t,dx) \bigg)
		\GS \int_{\DR^d} \bigg( \HA\TRACE\big( \DSC(t,dx) \big)
			+ \frac{1}{\gamma-1} \DSP(t,dx) \bigg)
\]
for all $t\in[0,\infty)$, as a consequence of
Definition~\ref{D:DISSI}~\eqref{L:COMPATIBILITY}/\eqref{L:NU}.
\end{remark}

%========== REMARK
\begin{remark}
One can prove that the boundedness of $\RHO$ in $\L^\infty([0,\infty);
\L^\gamma(\R^d))$ with Lipschitz continuity in a weaker topology (here: with
respect to the Wasserstein distance, which metrizes the weak* convergence of
measures) implies that $t \mapsto \RHO_t$ is continuous in time with respect to
the \emph{weak} $\L^\gamma(\R^d)$-topology. Similarly, we have that $t \mapsto
\BM_t$ is continuous with respect to the weak $\L^p(\R^d)$-topology.
\end{remark}

%========== PROPOSITION
\begin{proposition}\label{P:EXISTENCE}
For initial data as in \eqref{E:DATA}, dissipative solutions do exist.
\end{proposition}

%========== PROOF
\begin{proof}
We utilize the time variational time discretization in
\cite{CavallettiSedjroWestdickenberg2019}, with minor modifications that will be
pointed out below. The strategy is to generate an approximation of the solution
$(\RHO,\BM = \RHO\BU)$ of the isentropic Euler equations \eqref{E:IEE} at
discrete times $t_\tau^k := k\tau$, where $\tau>0$ and $k\in\N_0$, by solving a
convex minimization problem in each timestep. By interpolating in time and
passing to the limit $\tau\to 0$, we obtain a sequence of approximate solutions
that converge (up to a subsequence) towards a dissipative solution of
\eqref{E:IEE} in the sense of Definition~\ref{D:DISSI}. We will outline the main
steps below and refer the reader to \cite{CavallettiSedjroWestdickenberg2019}
for additional details.

\medskip

\textbf{Step~1.} Let us first describe the minimization problem, which is the
basis for the time discretization. It was shown in
\cite{CavallettiSedjroWestdickenberg2019} that for any given data
\[
	\RHO \in \SP_2(\R^d),
	\quad
	\BU \in \L^2(\R^d, \RHO)
\]
with finite energy, there exists a unique minimizer $X_\tau$ of the functional
\[
	\frac{3}{4\tau^2} \int_{\R^d} 
		  	\big|X(x)-\big( x+\tau\BU(x) \big) \big|^2 \,\RHO(dx)
		+ \int_{\R^d} U\big( r(x) \big) \det\big( \nabla X(x)^\S 
				\big)^{1-\gamma} \,dx,
\]
which is defined for $\R^d$-valued functions $X \in \L^2(\R^d, \RHO)$ that are
\emph{monotone}. Finiteness of the internal energy implies that $\RHO =: r
\LEB^d$, with $r$ some Lebesgue integrable function; recall
Definition~\ref{D:INT}. By slight abuse of notation, we will not always
distinguish between the measure $\RHO$ and its Lebesgue density $r$.
Monotonicity of $X_\tau$ implies $\BVS_\LOC$-regularity, and $\nabla X(x)^\S$ is
the symmetric part of the absolutely continuous part of the distributional
derivative of $X$, which is a locally finite $\MAT[\GS]{d}$-valued measure. The
optimality condition of this minimization problem takes the following form:
There exists $\LM_\tau \in \M\big( \R^d; \SYM[\GS]{d} \big)$ such that
\begin{equation}
	\int_{\R^d} \zeta(x) \cdot \frac{W_\tau(x)-\BU(x)}{\tau} r(x) \,dx
		= \int_{\R^d} \nabla\zeta(x) : 
			\Big( \PR_\tau(x) \,dx + \LM_\tau(dx) \Big)
\label{E:EULA}
\end{equation}
for all $\zeta \in \AF$ (see \eqref{E:AF}), with pressure field
\begin{equation}
	\PR_\tau(x) := P\big( r(x) \big)
			\det\big( \nabla X_\tau(x)^\S \big)^{1-\gamma}
				\; \big( \nabla X_\tau(x)^\S \big)^{-1}
\label{E:PTAU}
\end{equation}
(taking values in $\SYM[>]{d}$) and velocities
\begin{equation}
	W_\tau(x) := \frac{3}{2} V_\tau(x)-\frac{1}{2} \BU(x),
	\quad
	V_\tau(x) := \frac{X_\tau(x)-x}{\tau},
\label{E:VELOS}
\end{equation}
in $\L^2(\R^d, \RHO)$. If $\E[\RHO,\BU] := \int_{\R^d} \HA\RHO|\BU|^2 +
\INT[\RHO]$ denotes the total energy, then
\begin{align}
	\E[\RHO_\tau,\BU_\tau] 
		& + \int_{\R^d} {\TST\frac{1}{6}} |W_\tau(x)-\BU(x)|^2 \,\RHO(dx)
\label{E:ENINQ}\\
		& + \int_{\R^d} \Big( P\big( r(x) \big) 
			D_\INT\big( \nabla X_\tau(x)^\S-\ONE \big) \,dx
			+ \TRACE\big( \LM_\tau(dx) \big) \Big) \LS \E[\RHO,\BU],
\nonumber
\end{align}
where $\RHO_\tau := X_\tau \# \RHO$ and $\BU_\tau := W_\tau \circ X_\tau^{-1}$
belongs to $\L^2(\R^d, \RHO_\tau)$. Here $\#$ indicates the push-forward of
measures, defined as $\RHO_\tau(A) := \RHO(X_\tau^{-1}(A))$ for all Borel
subsets $A\subset\R^d$. One can show that $X_\tau$ can be extended by
monotonicity to a $\RHO$-essentially injective map on all of $\R^d$ so that
$\BU_\tau$ is indeed well-defined. 

The function $D_\INT$ is the Bregman divergence of the convex function $S
\mapsto \det(S)^{1-\gamma}$ with $S\in\SYM[>]{d}$. For any $\EPS>0$ there exists
$C_\EPS>0$ with
\begin{equation}
	\sup_{z\in\R^d, |z|=1} \Big| \Big\langle z, 
		\big( \ONE-\det(\ONE+S)^{1-\gamma} (\ONE+S)^{-1} \big) z 
			\Big\rangle \Big| \LS \EPS + C_\EPS D_\INT(S)
\label{E:DIFFE}
\end{equation}
for all $S \in \SYM{d}$ such that $\ONE+S$ is positive definite. 

We refer the reader to \cite{CavallettiSedjroWestdickenberg2019} for motivation,
proofs, and further discussion.

\medskip

\textbf{Step~2.} We now explain in which sense the minimizer of Step~1\
generates an approximate solution of \eqref{E:IEE} for \emph{one timestep}. We
define the interpolants
\begin{equation}
	X_t(x) := x + t V_\tau(X),
	\quad
	W_t(x) := \bigg( 1-\frac{t}{\tau} \bigg) \BU(x)
		+ \frac{t}{\tau} W_\tau(x)
\label{E:INTER}
\end{equation}
for $\RHO$-a.e.\ $x\in\R^d$ and $t\in[0,\tau]$. Then we interpolate density and
velocity as
\begin{equation}
	\RHO_t := X_t \# \RHO,
	\quad
	\BU_t := W_t \circ X_t^{-1}.
\label{E:RVINTER}
\end{equation}
This is well-defined because $X_t$ is $\RHO$-essentially injective.

For any $\eta \in \C^1(\R)$ and $\zeta \in \AF$, we now compute (integrating by
parts)
\begin{align}
	& -\int_0^\tau \eta'(t) \int_{\R^d} \zeta\big( X_t(x) \big) 
		\cdot W_t(x) \,\RHO(dx) \,dt
\label{E:MOMQ}\\
	& \qquad\qquad
		+\eta(\tau) \int_{\R^d} \zeta\big( X_\tau(x) \big) \cdot
			W_\tau(x) \,\RHO(dx)
		-\eta(0) \int_{\R^d} \zeta(x) \cdot \BU(x) \,\RHO(dx)
\nonumber\\
% 	& \qquad
% 		= \int_0^\tau \eta(t) \frac{d}{dt} \bigg( \int_{\R^d}
% 			\zeta\big( X_t(x) \big) \cdot W_t(x) \,\RHO(dx) \bigg) \,dt
% \nonumber\\
	& \qquad
		= \int_0^\tau \eta(t) \int_{\R^d} 
			\nabla\zeta\big( X_t(x) \big) : 
				\Big( W_t(x)\otimes \dot{X}_t(x) \Big) \,\RHO(dx) \,dt
\nonumber\\
	& \qquad\qquad
		+ \int_0^\tau \eta(t) \int_{\R^d}
				\zeta\big( X_t(x) \big) \cdot \dot{W}_t(x) \,\RHO(dx) \,dt.
\nonumber
\end{align}
Using definition \eqref{E:RVINTER}, we have that
\begin{align*}
	-\int_0^\tau \eta'(t) \int_{\R^d} \zeta\big( X_t(x) \big) \cdot W_t(x) 
			\,\RHO(dx) \,dt
		& = -\int_0^\tau \eta'(t) \int_{\R^d} \zeta(z) \cdot \BU_t(z) 
			\,\RHO_t(dz) \,dt.
\\
	\eta(\tau) \int_{\R^d} \zeta\big( X_\tau(x) \big) \cdot W_\tau(x)
			\,\RHO(dx)
		&= \eta(\tau) \int_{\R^d} \zeta(z) \cdot \BU_\tau(z) \,\RHO_\tau(dz).
\end{align*}

The first integral on the right-hand side of \eqref{E:MOMQ} can be rewritten as
\begin{align}
	& \int_0^\tau \eta(t) \int_{\R^d} 
		\nabla\zeta\big( X_t(x) \big) : 
			\Big( W_t(x)\otimes \dot{X}_t(x) \Big) \,\RHO(dx) \,dt
\label{E:LPOK}\\
	& \qquad
		= \int_0^\tau \eta(t) \int_{\DR^d}
			\nabla\zeta\big( X_t(x) \big) : \Big( W_t(x) \otimes
				\big( V_\tau(x)-W_t(x) \big) \Big) \,\RHO(dx) \,dt
\nonumber\\
	& \qquad\qquad
		+ \int_0^\tau \eta(t) \int_{\DR^d}
			\nabla\zeta\big( X_t(x) \big) : \Big( W_t(x) \otimes
				W_t(x) \Big) \,\RHO(dx) \,dt.
\nonumber
\end{align}
Because of \eqref{E:VELOS} and \eqref{E:INTER}, we have that
\[
	V_\tau(x) - W_t(x) = \bigg( \frac{2}{3}-\frac{t}{\tau} \bigg) 
		\big( W_\tau(x)-\BU(x) \big)
	\quad\text{for all $t\in[0,\tau]$.}
\]
By Young's inequality, for any $\EPS>0$ there exists a constant $C_\EPS>0$ with
\begin{align*}
	& \big\| W_t(x)\otimes \big( V_\tau(x)-W_t(x) \big) \big\|
\\
	& \qquad
		\LS \EPS
			\Bigg( \bigg( 1-\frac{t}{\tau} \bigg) |\BU(x)|^2 
				+ \frac{t}{\tau} |W_\tau(x)|^2 \Bigg) 
			+ C_\EPS \bigg| \frac{2}{3}-\frac{t}{\tau} \bigg| 
				|W_\tau(x) - \BU(x)|^2.
\end{align*}
It follows that
\begin{align}
	& \left| \int_0^\tau \eta(t) \int_{\R^d}
		\nabla\zeta\big( X_t(x) \big) : \Big( W_t(x) \otimes
			\big( V_\tau(x)-W_t(x) \big) \Big) \,\RHO(dx) \,dt \right|
\label{E:ERRK}\\
	& \qquad 
		\LS C\tau \; \Bigg\{  
			\EPS \int_{\R^d} \Big( |\BU(x)|^2 + |W_\tau(x)|^2 \Big) \,\RHO(dx)
				+ C_\EPS \int_{\R^d} |W_\tau(x)-\BU(x)|^2 \,\RHO(dx) \Bigg\},
\nonumber
\end{align}
where $C$ depends on the $\sup$-norms of $\eta$ and $\nabla\zeta$. The integral
multiplied by $\EPS$ can be bounded by the total energy, the one multiplied by
$C_\EPS$ by the energy dissipation; see \eqref{E:ENINQ}. The second integral on
the right-hand side of \eqref{E:LPOK} is
\begin{align*}
	& \int_0^\tau \eta(t) \int_{\R^d} \nabla\zeta\big( X_t(x) \big) :
			\Big( W_t(x) \otimes W_t(x) \Big) \,\RHO(dx) \,dt
\\
	& \qquad
		= \int_0^\tau \eta(t) \int_{\R^d} \nabla\zeta(z) :
			\Big( \BU_t(z) \otimes \BU_t(z) \Big) \,\RHO_t(dz) \,dt
\end{align*}
because of definition \eqref{E:RVINTER}.

The second integral on the right-hand side of \eqref{E:MOMQ} can be rewritten as
\begin{align}
	& \int_0^\tau \eta(t) \int_{\R^d} \zeta\big( X_t(x) \big) \cdot 
			\dot{W}_t(x) \,\RHO(dx) \,dt
\label{E:DEC1}\\
	& \qquad
		= \int_0^\tau \eta(t) \int_{\R^d} 
			\Big( \zeta\big( X_t(x) \big)-\zeta(x) \Big)
				\cdot \frac{W_\tau(x)-\BU(x)}{\tau} \,\RHO(dx) \,dt
\nonumber\\
	& \qquad\qquad
		+ \int_0^\tau \eta(t) \int_{\R^d} \zeta(x) \cdot 
			\frac{W_\tau(x)-\BU(x)}{\tau} \,\RHO(dx) \,dt.
\nonumber
\end{align}
Since $\zeta \in \AF$ we can estimate the first integral on the right-hand of
\eqref{E:DEC1} by
\begin{align}
	& \bigg| \int_0^\tau \eta(t) \int_{\R^d} 
			\Big( \zeta\big( X_t(x) \big)-\zeta(x) \Big)
				\cdot \frac{W_\tau(x)-\BU(x)}{\tau} \,\RHO(dx) \,dt \bigg|
\label{E:SHIFT}\\
	& \qquad
		\LS C \tau \; \Bigg\{
			\EPS \int_{\R^d} |V_\tau(x)|^2 \,\RHO(dx) 
				+ C_\EPS \int_{\R^d} |W_\tau(x)-\BU(x)|^2 \,\RHO(dx) \Bigg\}
\nonumber
\end{align}
for any $\EPS>0$ and a corresponding constant $C_\EPS$, by Cauchy-Schwarz and
Young's inequality. The constant $C$ depends on the $\sup$-norms of $\eta$ and
$\nabla\zeta$. The integral with $V_\tau$ is bounded by the total energy, while
the integral with $W_\tau-U$ is bounded by the energy dissipation; see
\eqref{E:ENINQ}. Notice that $V_\tau$ is a convex combination of $\BU$ and
$W_\tau$; see \eqref{E:VELOS}. Because of equation \eqref{E:EULA}, the last
integral in \eqref{E:DEC1} is
\begin{align}
	& \int_0^\tau \eta(t) \int_{\R^d} \zeta(x) \cdot \frac{W_\tau(x)-\BU(x)}{\tau}
			\,\RHO(dx) \,dt
\label{E:KOLP}\\
	& \qquad
		= \int_0^\tau \eta(t) \int_{\R^d} \nabla\zeta(x) : \Big( 
			\PR_\tau(x) \,dx + \LM_\tau(dx) \Big) \,dt.
\nonumber
\end{align}
We can now argue as in \cite{CavallettiSedjroWestdickenberg2019}. First, we
estimate
\begin{align}
	& \bigg| \int_0^\tau \eta(t) \int_{\R^d} \nabla\zeta(x) :
		\Big( \PR_\tau(x) -P\big( r(x) \big) \ONE \Big) \,dx \,dt \bigg|
\label{E:ERRL}\\
	& \qquad 
		\LS C\tau \; \Bigg\{  
			\EPS \int_{\R^d} U\big( r(x) \big) \,dx
				+ C_\EPS \int_{\R^d} P\big( r(x) \big) 
					D_\INT\big( \nabla X_\tau(x)^\S-\ONE \big) \,dx
						\Bigg\},
\nonumber
\end{align}
using \eqref{E:PTAU}/\eqref{E:DIFFE} and the fact that $P = (\gamma-1)U$ for
polytropic gases. Again the first integral on the right-hand side can be bounded
by the total energy, the second one by the energy dissipation; see
\eqref{E:ENINQ}. Notice that since $\PR_\tau, \ONE$ are symmetric, only the
symmetric part of $\nabla\zeta(x)$ is relevant for the estimate, therefore
\eqref{E:DIFFE} can indeed be used. Finally, the last integral in \eqref{E:KOLP}
can be estimated as
\begin{equation}
	\bigg| \int_0^\tau \eta(t) \int_{\R^d} \nabla\zeta(x) :
			\LM_t(dx) \,dt \bigg|
		\LS C \tau \int_{\R^d} \TRACE\big( \LM_t(dx) \big),
\label{E:RTAU}
\end{equation}
and the integral can be bounded by the energy dissipation; see \eqref{E:ENINQ}.
We have used that $\LM_t$ takes values in positive semidefinite matrices, so
the trace is equivalent to the Frobenius norm. The constant $C$ in
\eqref{E:ERRL} and \eqref{E:RTAU} depends on the $\sup$-norms of $\eta$ and
$\nabla\zeta$. Collecting all terms, we obtain the identity
\begin{align}
	& -\int_0^\tau \eta'(t) \int_{\R^d} 
		\zeta(x) \cdot \BU_t(x) \,\RHO_t(dx) \,dt
\label{E:MOMQ2}\\
	& \qquad\qquad
		+\eta(\tau) \int_{\R^d} \zeta(x) \cdot \BU_\tau(x) \,\RHO_\tau(dx)
		-\eta(0) \int_{\R^d} \zeta(x) \cdot \BU(x) \,\RHO(dx)
\nonumber\\
	& \qquad
		= \int_0^\tau \eta(t) \int_{\R^d} 
			\nabla\zeta(z) :  
				\Big( \BU_t(z)\otimes\BU_t(z) \Big) \,\RHO_t(dz) \,dt
\nonumber\\
	& \qquad\qquad
		+ \int_0^\tau \eta(t) \int_{\R^d}
				P\big( r(x) \big) \nabla\cdot\zeta(x) \,dx \,dt
			\quad + \quad \text{ERROR}.
\nonumber
\end{align}
Note that the spatial integral in \eqref{E:MOMQ2} involving $P(r(x))$ \emph{does
not depend on $t$}. The $\text{ERROR}$ term collects the contributions from
\eqref{E:ERRK}, \eqref{E:SHIFT}, and \eqref{E:ERRL}. Thus
\begin{equation}
	|\text{ERROR}|
		\LS C\tau \; \bigg( \EPS \E[\RHO,\BU] 
			+ C_\EPS \Big( \E[\RHO,\BU]-\E[\RHO_\tau,\BU_\tau] \Big) \bigg).
\label{E:ERROR}
\end{equation}

We can now iterate the procedure outlined above, using the final data
$(\RHO_\tau, \BU_\tau)$ of one timestep as the initial data for the next
minimization problem. Notice that because of \eqref{E:ENINQ} the new initial
data again has finite total energy and is therefore admissible. We will denote
by $(\RHO_\tau^k, \BU_\tau^k)$ and $X_\tau^k, W_\tau^k$ the corresponding
approximate solutions and transport maps/velocities at times $t_\tau^k :=
k\tau$, with $k\in\N_0$. Here we use the subscript $\tau$ to emphasize the
\emph{dependence on the timestep}. Let
\begin{equation}
	X_{\tau,t}(x) := \frac{t_\tau^{k+1}-t}{\tau} x 
		+ \frac{t-t_\tau^k}{\tau} X_\tau^k(x),
	\quad
	W_{\tau,t}(x) := \frac{t_\tau^{k+1}-t}{\tau} \BU_\tau^k(x)
		+ \frac{t-t_\tau^k}{\tau} W_\tau^k(x)
\label{E:INTR2}
\end{equation}
for $\RHO_\tau^k$-a.e.\ $x\in\R^d$ be the interpolation of transport/velocity,
and 
\begin{equation}
	\RHO_{\tau,t} := X_{\tau,t} \# \RHO_\tau^k,
	\quad
	\BU_{\tau,t} := W_{\tau,t} \circ X_{\tau,t}^{-1}
\label{E:INTR3}
\end{equation}
the interpolated density and Eulerian velocity, for $t \in [t_\tau^k,
t_\tau^{k+1}]$ and $k\in\N_0$.

If now $\eta \in \C^1_c([0,\infty))$ and $\zeta \in \AF$, then
$(\RHO_\tau,\BU_\tau)$ satisfies the analogue of \eqref{E:MOMQ2}, with time
integration over $[0,\infty)$. When adding up the error terms, we find that the
contributions from the first term on the right-hand side of \eqref{E:ERROR}
amounts to adding $\tau$ for every timestep because $\E[\RHO_\tau^k,
\BU_\tau^k]$ is bounded by $\bar{E}$ uniformly in $k,\tau$; see \eqref{E:ENINQ}.
Since $\eta$ has compact support in some interval $[0,T]$, this contibution is
bounded by $CT \, \EPS \bar{E}$, which can be made arbitrarily small, uniformly
in $\tau$, by choosing $\EPS$ small. On the other hand, the contributions from
the second term on the right-hand side of \eqref{E:ERROR} can be summed up over
all timesteps because the dissipation terms form a telescope sum. This
contribution is therefore bounded by $C\tau \, C_\EPS \bar{E}$, which converges
to zero as $\tau\to 0$. We conclude that the error term in the momentum equation
vanishes in the limit. The continuity equation is satisfied, by construction,
but with respect to a different velocity field, which is obtained by
transporting the piecewise in time constant velocity $V_\tau$ along the
transport map; see \cite{CavallettiSedjroWestdickenberg2019} for details.

\medskip

\textbf{Step~3.} We will now pass to the limit $\tau\to 0$. The main issue
here is to define suitable Young measures that can represent the weak limits of
non-linear terms of the form $\RHO\BU\otimes\BU$ and $P(\RHO)$. We will use two
such Young measures: First, let
\[
	\int_{\DR^d\times\BFX} \phi(x,\FX) \,\epsilon_\tau(t, dx, d\FX) 
		:= \int_{\R^d} \phi\big( x, r_\tau^k(x), m_\tau^k(x) \big)
			\; h\big( r_\tau^k(x), m_\tau^k(x) \big) \,dx
\]
for all $\phi \in \C(\DR^d\times\BFX)$ and $t\in [t_\tau^k, t_\tau^{k+1})$,
$k\in\N_0$, where
\[
	\RHO_\tau^k =: r_\tau^k \LEB^d,
	\quad
	m_\tau^k := r_\tau^k \BU_\tau^k.
\]
Notice that $\epsilon_\tau$ is \emph{piecewise constant in time}. Its total
variation is given by
\[
	\|\epsilon_\tau(t,\cdot)\|_{\M(\DR^d\times\BFX)}
		= \int_{\DR^d\times\BFX} \epsilon_\tau(t, dx, d\FX)
		= 1 + \E[\RHO_\tau^k, \BU_\tau^k]
\]
for $t\in [t_\tau^k, t_\tau^{k+1})$, which is bounded by the initial total
energy $\bar{E}$ because of \eqref{E:ENINQ}, uniformly in $k, \tau$. Indeed we
have more: The function
\[
	t \mapsto E_\tau(t) := \|\epsilon_\tau(t,\cdot)\|_{\M(\DR^d\times\BFX)} - 1,
\]
which equals the total energy, is \emph{non-increasing in time} with $E_\tau(0)
= \bar{E}$. For any sequence $\tau_n \longrightarrow 0$ as $n\to\infty$, there
now exists a subsequence (not relabeled, for simplicity) and a Young measure
$\epsilon \in \BBE^*$ such that $\epsilon_{\tau_n} \WEAK \epsilon$ weak* in
$\BBE^*$ as $n\to\infty$ (which means testing against functions in $\BBE$), by
Banach-Alaoglu theorem. Because of Helly's selection theorem, we may extract
another subsequence if necessary and obtain that $E_{\tau_n}(t) \longrightarrow
E(t)$ pointwise for all $t\in[0,\infty)$, with total enery
\[
	t \mapsto E(t) := \|\epsilon(t,\cdot)\|_{\M(\DR^d\times\BFX)} - 1
\]
non-increasing in time and bounded by $\bar{E}$.

For the second Young measure we use a \emph{continuous time interpolation}: Let
\[
	\int_{\DR^d\times\BFX} \phi(x,\FX) \,\nu_\tau(t, dx, d\FX) 
		:= \int_{\R^d} \phi\big( x, r_{\tau,t}(x), m_{\tau,t}(x) \big)
			\; h\big( r_{\tau,t}(x), m_{\tau,t}(x) \big) \,dx
\]
for all $\phi \in \C(\DR^d\times\BFX)$ and $t\in [0,\infty)$, with
$(\RHO_{\tau,t}, \BU_{\tau,t})$ defined in \eqref{E:INTR3}/\eqref{E:INTR2},
\begin{equation}
	\RHO_{\tau,t} =: r_{\tau,t} \LEB^d,
	\quad
	m_{\tau,t} := r_{\tau,t} \BU_{\tau,t}.
\label{E:TAUT}
\end{equation}
Again we must establish uniform boundedness. The map
\[
	t \mapsto \int_{\R^d} \HA |\BU_{\tau,t}(z)|^2 \,\RHO_{\tau,t}(dx)
		= \int_{\R^d} \HA |W_{\tau,t}(x)|^2 \,\RHO_\tau^k(dx)
\]
is convex for $t \in [t_\tau^k, t_\tau^{k+1}], k\in\N_0$; see \eqref{E:INTR2}.
Similarly, the map
\[
	t \mapsto \INT[X_{\tau,t}| \RHO_{\tau,t}] 
		:= \int_{\R^d} U\big( r_\tau^k(x) \big)
			\det\big( \nabla X_{\tau,t}(x)^\S \big)^{1-\gamma} \,dx
\]
is convex in each timestep; see the proof of Proposition~5.23 in
\cite{CavallettiSedjroWestdickenberg2019}. Thus
\[
	t  \mapsto \int_{\R^d} \HA |\BU_{\tau,t}(z)|^2 \,\RHO_{\tau,t}(dx)
		+ \INT[X_{\tau,t}| \RHO_{\tau,t}]
	\quad\text{with $t\in [t_\tau^k, t_\tau^{k+1}]$}
\]
is convex. In fact, it is bounded by the total energy at time $t_\tau^k$; see
Proposition~5.23 in \cite{CavallettiSedjroWestdickenberg2019}. Since
$\INT[\RHO_{\tau,t}] \LS \INT[X_{\tau,t}| \RHO_{\tau,t}]$ (see Remark~5.14 in
\cite{CavallettiSedjroWestdickenberg2019}) it follows that
\begin{equation}
	\E[\RHO_{\tau,t}, \BU_{\tau,t}] \LS \E[\RHO_\tau^k, \BU_\tau^k]
	\quad\text{for $t\in[t_\tau^k, t_\tau^{k+1}], k\in\N_0$.}
\label{E:STEPBD}
\end{equation}
In particular, the total variation of $\nu_{\tau,t}$ can be bounded by the total
variation of $\epsilon_{\tau,t}$ for all $t\in[0,\infty)$, which in the limit
$\tau\to 0$ will produce Property~\eqref{L:NU} of Definition~\ref{D:DISSI}.
Extracting another subsequence of $\tau_n\to 0$ if necessary (not relabeled),
we may assume that $\nu_{\tau_n} \WEAK \nu$ weak* in $\BBE^*$ as $n\to\infty$,
for some $\nu \in \BBE^*$.

Moreover, since the curves $t \mapsto (\RHO_{\tau,t}, \BM_{\tau,t})$ are
Lipschitz continuous with respect to the Wasserstein distance and the bounded
Lipschitz norm, respectively, uniformly in $\tau$, one can use the Arzelà-Ascoli
theorem to conclude the existence of another subsequence (not relabeled) and of
limit density/momentum $(\RHO,\BM)$ such that
\[
	(\RHO_{\tau_n,t}, \BM_{\tau_n,t}) \longrightarrow (\RHO_t, \BM_t)
	\quad\text{weak* in the sense of measures,}
\]
for all $t\in [0,\infty)$. We refer the reader to Lemmas~6.1 and 6.2 in
\cite{CavallettiSedjroWestdickenberg2019}; see also the proof of
Lemma~\ref{L:CLOSED} for a similar argument. Then lower semicontinuity of the
internal energy (see Definition~\ref{D:INT}) and the fact the the total energy
stays uniformly bounded imply that the limit density $\RHO_t$ is again
absolutely continuous with respect to the Lebesgue measure. Indeed since $U(r) =
\kappa r^\gamma$ with $\gamma>1$ has superlinear growth at infinity, the
internal energy would be infinite if the density had any singular parts.
Similarly, lower semicontinuity of the kinetic energy forces $\BM_t$ to be
absolutely continuous with respect to $\RHO_t$, and thus in turn with respect to
the Lebesgue measure; see again the proof of Lemma~\ref{L:CLOSED} for a similar
argument.

\medskip

\textbf{Step~4.} We now establish the decomposition in
Definition~\ref{D:DISSI}~\eqref{L:DECOMPOSITION}. We consider only the Young
measure $\nu$ since the argument for $\epsilon$ is analogous. For any $\tau>0$,
we define the weakly measurable map $\upsilon^\tau \colon [0,\infty) \times \R^d
\longrightarrow \SP(X_V)$ by 
\[
	\int_0^\infty \int_{\R^d} \phi(t,x) \int_{X_V} \varphi(\FX) 
			\,\upsilon^\tau_{t,x}(d\FX) \,dx \,dt 
		:= \int_0^\infty \int_{\R^d} \phi(t,x) \varphi\big(
			r_{\tau,t}(x), m_{\tau,t}(x) \big) \,dx \,dt
\]
for all $\phi \in \C_c([0,\infty)\times\R^d)$ and $\varphi \in \C_c(X_V)$; see
\eqref{E:TAUT} and Remark~\ref{R:VAC} for notation. As was explained there, the
set $X_V$ is topologically equivalent to $[0,\infty)\times \R^d$, where $V$ is
identified with $\{0\}\times \R^d$. Thus $\upsilon^\tau_{t,x}$ is the Dirac
measure at $(r_{\tau,t}(x),m_{\tau,t}(x))$; in vaccuum it assigns mass one to
the point in $V$ corresponding to the velocity $\BU_{\tau,t}(x)$. Since
$\upsilon^\tau_{(t,x)}$ is a probability measure for a.e.\ $(t,x)$ and all
$\tau$, we have that 
\[
	\int_A \upsilon^\tau_{t,x}(X_V) \,dx \,dt = \LEB^{d+1}(A)
	\quad\text{for all $A \subset [0,\infty)\times \R^d$ Borel.}
\]
Consider now the compact sets $\Omega_N := [0,N] \times \bar{B}_N(0)$ with
$N\in\N$. Then 
\begin{align*}
	\sup_{\tau>0} \int_{\Omega_N} \int_{X_V} h(\FX) 
			\,\upsilon^\tau_{t,x}(d\FX)	\,dx \,dt
		& \LS \sup_{\tau>0} \int_0^N \int_{\R^d} \Bigg(
			\RHO_\tau + \bigg( \frac{|\BM_\tau|^2}{2\RHO_\tau} + U(\RHO_\tau) 
				\bigg) \Bigg)
\\
		& \vphantom{\int}
			\LS N (1+\bar{E}) < +\infty,
\end{align*}
with $h$ extended to a lower semicontinuous and convex map
\[
	(\RHO,\BM) \mapsto \begin{cases}
		\DST \RHO+\bigg( \frac{|\BM|^2}{2\RHO}+U(\RHO) \bigg) 
			& \text{if $(\RHO,\BM) \in (0,\infty)\times\R^d$}
\\
		0 & \text{if $(\RHO,\BM) \in V$}
\\
		+\infty & \text{otherwise}
	\end{cases}
\]
It follows that the map $(\RHO,\BM) \mapsto h(\RHO,\BM)$ is $\inf$-compact,
i.e., its sublevel sets are compact. We can then apply Theorem~4.3.2 in
\cite{CastaingRaynaudDeFitteValadier2004} to conclude that the family
$\{\upsilon^\tau\}_\tau$, when restricted to $\Omega_N \times X_V$, is
\emph{strictly tight} and therefore (sequentially) relatively compact with
respect to the topology of $S$-stable convergence; see Section~2.1 and
Theorem~4.3.5 in \cite{CastaingRaynaudDeFitteValadier2004}. In particular, if
$\tau_n\to 0$ is the sequence of timesteps that generates the Young measures
$\nu, \epsilon$ of Step~3, by a diagonal argument with $N\to\infty$, we obtain a
weakly measurable map $\upsilon \colon [0,\infty)\times \R^d \longrightarrow
\SP(X_V)$ such that
\begin{align}
	& \int_0^\infty \int_{\R^d} \phi(t,x) \varphi\big(
		r_{\tau_n,t}(x), m_{\tau_n,t}(x) \big) \,dx \,dt
\label{E:COVO}\\
	& \qquad
		\longrightarrow \int_0^\infty \int_{\R^d} \phi(t,x) 
			\int_{X_V} \varphi(\FX) \,\upsilon_{t,x}(d\FX) \,dx \,dt 
	\quad\text{as $n\to\infty$}
\nonumber
\end{align}
for all $\phi \in \C_c([0,\infty)\times\R^d)$ and $\varphi\in\C_c(X_V)$, along
suitable subsequences that are not relabeled, for simplicity. Using a similar
argument, we also obtain
\begin{equation}
\begin{gathered}
	\RHO_t(dx) 
		= \int_{X_V} \RHO \,\upsilon_{t,x}\big( d(\RHO,\BM) \big) \,dx,
\\
	\BM_t(dx) 
		= \int_{X_V} \BM \,\upsilon_{t,x}\big( d(\RHO,\BM) \big) \,dx
\end{gathered}
\label{E:IDES}
\end{equation}
for a.e.\ $t$. Recall that $\RHO_t, \BM_t$ are absolutely continuous with
respect to the Lebesgue measure $\LEB^d$. By convexity and Jensen's inequality,
it follows that
\begin{equation}
\begin{aligned}
	P\big( r_t(x) \big)
		& \LS \int_{X_V} P(\RHO) \,\upsilon_{t,x}\big( d(\RHO,\BM) \big),
\\
	\frac{m_t(x)\otimes m_t(x)}{r_t(x)}
		& \LS \int_{X_V} \frac{\BM\otimes\BM}{\RHO} 
			\,\upsilon_{t,x}\big( d(\RHO,\BM) \big)
\end{aligned}
\label{E:FIRS}
\end{equation}
in the sense of symmetric matrices, with $\RHO_t = r_t \LEB^d$ and $m_t :=
r_t\BU_t$, for a.e.\ $(t,x)$. Because of Step~3 and \eqref{E:COVO}, we observe
that for all $\varphi \in \C_c(X_V)$ and a.e.\ $t$
\begin{equation}
	\int_{X_V} \varphi(\FX) \,\upsilon_{t,x}(d\FX) \,dx
		= \int_{X_V} \frac{\varphi(\FX)}{h(\FX)} \,\nu_t(dx, d\FX).
\label{E:EQUL}
\end{equation}
Here we have used that $\varphi$ vanishes on $\BFX\setminus X_V$, so it is
sufficient to integrate over $X_V$ in \eqref{E:EQUL} instead of $\BFX$. By
monotone convergence (recall that $\nu_t$ has finite total variation), we
generalize \eqref{E:EQUL} to $\varphi \in \C(X_V)$ with $|\varphi(\FX)| \LS C
h(\FX)$ for all $\FX \in X$, where $C$ is some constant. From this, we obtain
the inequalities
\begin{equation} 
\begin{gathered}
	\int_{X_V} P(\RHO) \,\upsilon_{t,x}\big( d(\RHO,\BM) \big) \,dx
		\LS \int_{\BFX} \frac{P(\RHO)}{h(\RHO,\BM)} 
			\,\nu_t\big(dx, d(\RHO,\BM) \big),
\\
	\int_{X_V} \frac{\BM\otimes\BM}{\RHO} 
			\,\upsilon_{t,x}\big( d(\RHO,\BM) \big) \,dx
		\LS \int_{\BFX} \frac{\BM\otimes\BM/\RHO}{h(\RHO,\BM)} 
			\,\nu_t\big(dx, d(\RHO,\BM) \big)
\end{gathered}
\label{E:EQLL}
\end{equation}
in the sense of symmetric matrices. Notice that on the right-hand side of
\eqref{E:EQLL} we included the contributions that the Young measure $\nu_t$ may
have ``at infinity'', i.e., in the set $\BFX \setminus X_V$. Combining
\eqref{E:FIRS} and \eqref{E:EQLL}, we have
\begin{equation}
\begin{aligned}
	P\big( r_t(x) \big) \,dx 
		& \LS \lbrack P(\RHO) \rbrack(t,dx),
\\
	r_t(x) \BU_t(x)\otimes\BU_t(x) \,dx 
		& \LS \lbrack \BM\otimes\BM/\RHO \rbrack(t,dx)
\end{aligned}
\label{E:LSS}
\end{equation}
in the sense of symmetric matrices. The differences between the left- and
right-hand sides of \eqref{E:LSS} define the defect measures $\DSC, \DSP$ in
Definition~\ref{D:DISSI}~\eqref{L:DECOMPOSITION}. For the second Young measure
$\epsilon$ and $\llbracket \cdot \rrbracket$ the argument works the same.

We can then argue as in \cite{CavallettiSedjroWestdickenberg2019} to finish the
proof of Proposition~\ref{P:EXISTENCE}.
\end{proof}

%========== REMARK
\begin{remark}\label{R:CHANGES} In \cite{CavallettiSedjroWestdickenberg2019} a
different time interpolation was used: The velocity was updated to the value
$W_\tau$ at the beginning of the timestep, then transported constantly. As a
consequence, the total energy could jump upward first, before decreasing
continuously to its value at the end of the timestep. Here we use the piecewise
linear interpolation \eqref{E:INTR2} instead, which gives us the energy
inequality \eqref{E:STEPBD}. The necessary changes in the derivation of the
momentum equation have been outlined in Step~2 above. The proof of Lipschitz
continuity of the momentum $t \mapsto \BM_t$ in Lemma~6.2 in
\cite{CavallettiSedjroWestdickenberg2019} must also be adapted, taking into
account the error estimate \eqref{E:SHIFT}. Note that \eqref{E:SHIFT} is already
uniform in the $\BL(\R^d)$-norm of the test function. Indeed, since the constant
$C$ on the right-hand side of \eqref{E:SHIFT} only depends on the $\sup$-norm of
$\nabla\zeta$, one can take the supremum over all $\zeta \in \BL_1(\R^d)$ to
conclude that the error is small in the bounded Lipschitz norm. We omit the
details. Different from \cite{CavallettiSedjroWestdickenberg2019}, here we use a
sum in the definition of the $\BL(\R^d)$-norm \eqref{E:BKNORM}, not a $\max$.
Both forms are equivalent. Moreover, we do not assume that the initial total
momentum vanishes.
\end{remark}

%%%%%%%%%%%%%%%%%%%%%%%%%%%%%%%%%%%%%%%%%%%%%%%%%%%%%%%%%%%%%%%%%%%%%%%%%%%%%%%%
%%%%%%%%%%%%%%%%%%%%%%%%%%%%%%%%%%%%%%%%%%%%%%%%%%%%%%%%%%%%%%%%%%%%%%%%%%%%%%%%
%%%%%%%%%%%%%%%%%%%%%%%%%%%%%%%%%%%%%%%%%%%%%%%%%%%%%%%%%%%%%%%%%%%%% Properties

\subsection{Properties}

The following lemma makes precise in which sense the quantity \eqref{E:INTRACE}
is in fact a measure of the acceleration of the dissipative solution.

%========== LEMMA
\begin{lemma}\label{L:MOMS} Suppose $(\epsilon, \nu)$ is a dissipative solution
of the isentropic Euler equations \eqref{E:IEE}, as introduced in
Definition~\ref{D:DISSI}. Let $t\mapsto\BM_t$ be the momentum associated to
$(\epsilon, \nu)$. With $a(t|\epsilon, \nu)$ defined by \eqref{E:INTRACE}, we
have that 
\[
	a(t|\epsilon, \nu) = |\BM'|(t)
	\quad\text{for a.e.\ $t \in [0,\infty)$,}
\] 
where $|\BM'|$ is the metric derivative of $\BM$ induced by the distance
\eqref{E:DULIP}.
\end{lemma}

%========== PROOF
\begin{proof}
For any $\eta \in \C^1_c((0,\infty))$ and $\zeta \in \C^1_c(\R^d; \R^d)$ we have
that
\begin{equation}
	-\int_0^\infty \eta'(t) \int_{\R^d} \zeta(x) \cdot \BM_t(dx) \,dt
		= \int_0^\infty \eta(t) \int_{\R^d} \nabla\zeta(x) : \FLUX(t,dx) \,dt
\label{E:OPP}
\end{equation}
because of the momentum equation; see Property~\eqref{L:CONSLAWS} of
Definition~\ref{D:DISSI}. Here 
\[
	\FLUX(t,dx) = \lbrack \BM\otimes\BM/\RHO 
		\rbrack(t,dx) + \llbracket P(\RHO) \rrbracket(t,dx) \ONE
\]
is the momentum flux. In \eqref{E:OPP} it suffices to integrate over $\R^d$ on
the right-hand side (instead of the compactification $\DR^d$) because $\zeta$
vanishes at infinity. Consider now a sequence of test functions $\eta^k$ that
converges pointwise a.e.\ to the characteristic function of some time interval
$[s,t]$ with $0\LS s\LS t$. Since the map $r \mapsto \BM_r$ is Lipschitz
continuous with respect to the bounded Lipschitz norm (testing against $\zeta
\in \BL(\R^d;\R^d)$), we can pass to the limit $k \to \infty$ in \eqref{E:OPP}
to obtain
\begin{equation}
	-\int_{\R^d} \zeta(x) \cdot \big( \BM_t(dx)-\BM_s(dx) \big)
		= \int_s^t \int_{\R^d} \nabla\zeta(x) : \FLUX(r,dx) \,dr.
\label{E:IDQ}
\end{equation}
With $\|\cdot\|$ the operator norm induced by the Euclidean norm $|\cdot|$ on
$\R^d$, we have that
\[
	\|\zeta\|_{\LIP(\R^d)} \LS 1
	\quad\Longrightarrow\quad
	\|\nabla\zeta(x)\| \LS 1
	\quad\Longrightarrow\quad
	\xi\cdot\nabla\zeta(x)\xi \LS |\xi|^2
\]
for all $x, \xi \in \R^d$. Recall that $\zeta$ is differentiable, by assumption.
Then
\begin{equation}
	\left| \int_{\R^d} \zeta(x) \cdot \big( \BM_t(dx)-\BM_s(dx) \big) \right|
		\LS \int_s^t \int_{\R^d} \TRACE\big( \FLUX(r,dx) \big) \,dr. 
\label{E:TRIN}
\end{equation}
Recall that $\FLUX$ takes values in the symmetric, positive semi-definite
matrices so that its trace is non-negative and we can change the domain of
integration to $\DR^d$.

We claim that \eqref{E:TRIN} remains true for all maps $\zeta \colon \R^d
\longrightarrow \R^d$ with $\|\zeta\|_{\LIP(\R^d)} \LS 1$.  Assuming for the
moment that the claim is true, we take the supremum on both sides of
\eqref{E:TRIN} over such $\zeta$ and conclude that 
\begin{equation}
	d(\BM_t,\BM_s) \LS \int_s^t \int_{\DR^d} \TRACE\big( \FLUX(r,dx) \big) \,dr
	\quad\text{for all $0\LS s\LS t$;}
\label{E:ONESIDE}
\end{equation}
see \eqref{E:DULIP}. In particular, the map $t\mapsto a(t|\epsilon,\nu)$ defined
in \eqref{E:INTRACE} is an \emph{upper bound} of the metric derivative $|\BM'|$
for all times. On the other hand, we can use the test function $\zeta = \ID$ in
the momentum equation because $\ID \in \AF$. This gives 
\begin{equation}
	- \int_0^\infty \eta'(t) \int_{\R^d} x \cdot \BM_t(dx) \,dt
		= \int_0^\infty \eta(t) \int_{\DR^d} \TRACE\big( \FLUX(t,dx) \big) \,dt
\label{E:OPP2}
\end{equation}
Note that we integrate over $\DR^d$ on the right-hand side of \eqref{E:OPP2}.
Considering again a sequence of test functions $\zeta^k$ converging pointwise
a.e.\ to the characteristic function of $[s,t]$, we want to pass to the limit
$k\to\infty$ on either side of \eqref{E:OPP2}. This time we may not have
Lipschitz continuity of the map $t \mapsto \int_{\R^d} x\cdot \BM_t(dx)$ but
this map is still in $\L^\infty([0,\infty))$. Then the Lebesgue differentiation
theorem implies
\begin{equation}
	-\int_{\R^d} x \cdot \big( \BM_t(dx)-\BM_s(dx) \big)
		= \int_s^t \int_{\DR^d} \TRACE\big( \FLUX(r,dx) \big) \,dr
\label{E:PLO}
\end{equation}
for \emph{almost every} $0\LS s\LS t$. Estimating the left-hand side of
\eqref{E:PLO} from above by taking the supremum over all Lipschitz continuous
test functions with $\|\zeta\|_{\LIP(\R^d)} \LS 1$ (of which $\zeta = \ID$ is
one), we obtain the opposite inequality to \eqref{E:ONESIDE}, and therefore
equality for generic $0\LS s\LS t$. It only remains to prove the claim made
above.

To this end, consider $\varphi \in \C^1_c(\R^d) $ with $\varphi(\R^d)
\subset [0,1]$ and 
\[
	\text{$\varphi(x) = 1$ if $|x|\LS 1$,}
	\quad
	\text{$\varphi(x) = 0$ if $|x|\GS 2$.}
\]
For any $R, \EPS > 0$ we define the rescaled cut-off function/mollifier
\[
	\eta_R(x) := \varphi(x/R),
	\quad
	\varphi_\EPS(x) := \EPS^{-d} \varphi(x/\EPS)
\]
where $x\in \R^d$. For given $\zeta \colon \R^d \longrightarrow \R^d$ with
$\|\zeta\|_{\LIP(\R^d)} \LS 1$ let 
\[
	\zeta_\EPS := \zeta\star\varphi_\EPS
	\quad\text{and}\quad
	\zeta_{R,\EPS} := \eta_R \zeta_\EPS,
\]
with the latter being an element of $\C^1_c(\R^d)$ for every $R, \EPS > 0$. We
still have \eqref{E:IDQ} with $\zeta_{R,\EPS}$ in place of $\zeta$. We decompose
the right-hand side in the form
\begin{align}
	\int_s^t \int_{\R^d} \nabla\zeta_{R,\EPS}(x) : \FLUX(r,dx) \,dr
		& = \int_s^t \int_{\R^d} \frac{1}{R} \nabla\eta\Big( \frac{x}{R} \Big)
			\otimes \zeta_\EPS(x) : \FLUX(r,dx) \,dr
\nonumber\\
		& + \int_s^t \int_{\R^d} \eta_R(x) \nabla\zeta_\EPS(x) 
			: \FLUX(r,dx) \,dr.
\label{E:POLL}
\end{align}
By the properties of mollifications, the assumption on $\zeta$ implies
$\|\nabla\zeta_\EPS \|_{\L^\infty(\R^d)} \LS 1$ so that $|\zeta_\EPS(x)| \LS A +
|x|$ for all $x\in \R$, with $A\GS 0$ some constant. Then
\[
	\bigg| \int_{\R^d} \frac{1}{R} \nabla\eta\Big( \frac{x}{R} \Big) 
		\otimes \zeta_\EPS(x) : \FLUX(r,dx) \bigg| 
		\LS C \frac{A+2R}{R} \int_{\R^d} \ONE_{\complement B_R(0)}(x)
			\,\TRACE\big( \FLUX(r,dx) \big),
\]
with $C$ a constant depending only on $\nabla\eta$, for any fixed $r$. Note that
this estimate is uniform in $\EPS$. As the integrand converges pointwise to zero
as $R\to\infty$, the integral vanishes in the limit, by dominated convergence.
Moreover, since the $x$-integral over $\TRACE(\FLUX)$ is bounded uniformly in
time, we can pass to the limit $R\to\infty$ also in the corresponding time
integral in \eqref{E:POLL}, using dominated convergence once again. In a similar
fashion, we argue that the second integral on the right-hand side of
\eqref{E:POLL} is bounded by $\int_s^t \int_{\R^d} \TRACE(\FLUX(r,dx)) \,dr$,
uniformly in $\EPS, R$. This gives the right-hand side of \eqref{E:TRIN}. To
pass to the limit on the left-hand side of \eqref{E:TRIN}, we first let $\EPS\to
0$ with $R$ fixed. Since $\BM_t$ is absolutely continuous with respect to the
Lebesgue measure and, in fact, in $\L^p(\R^d)$ for some $p>1$ (see
\eqref{E:SPACES}), we conclude that
\[
	\int_{\R^d} \eta_R(x) \zeta_\EPS(x)\cdot \BM_t(dx)
		\longrightarrow \int_{\R^d} \eta_R(x) \zeta(x)\cdot \BM_t(dx)
	\quad\text{as $\EPS\to 0$,}
\]
using the properties of mollifications. Moreover, since $|\zeta(x)| \LS A + |x|$
for all $x\in\R^d$ and some constant $A\GS 0$, and since $\BM_t$ has finite
first moment, we can again use the dominated convergence theorem to obtain 
\[
	\int_{\R^d} \eta_R(x) \zeta(x)\cdot \BM_t(dx) 
		\longrightarrow \int_{\R^d} \zeta(x) \cdot \BM_t(dx)
	\quad\text{as $R\to\infty$.}
\]
For the integral involving $\BM_s$ we argue analogously.
\end{proof}

%%%%%%%%%%%%%%%%%%%%%%%%%%%%%%%%%%%%%%%%%%%%%%%%%%%%%%%%%%%%%%%%%%%%%%%%%%%%%%%%
%%%%%%%%%%%%%%%%%%%%%%%%%%%%%%%%%%%%%%%%%%%%%%%%%%%%%%%%%%%%%%%%%%%%%%%%%%%%%%%%
%%%%%%%%%%%%%%%%%%%%%%%%%%%%%%%%%%%%%%%%%%%%%%%%%%%%%%%%%%%%%%%%%%%%%%%%%%%%%%%%
%%%%%%%%%%%%%%%%%%%%%%%%%%%%%%%%%%%%%%%%%%%%%%%%%%%%%%%%%%%%%%%%%%%%%%%%%%%%%%%%
%%%%%%%%%%%%%%%%%%%%%%%%%%%%%%%%%%%%%%%%%%%%%%%%%%%%%%%%%%% Minimal Acceleration

\section{Minimal Acceleration}
\label{S:MA}

For initial data $(\bar{\RHO}, \bar{\BM})$ as in \eqref{E:DATA}, we define the
set
\begin{equation}
	S := \big\{ \text{$(\epsilon, \nu)$ dissipative solution 
		of \eqref{E:IEE} with initial data $(\bar{\RHO}, \bar{\BM})$} \big\},
\label{E:SINIT}
\end{equation}
which is non-empty; see Proposition~\ref{P:EXISTENCE}. We now introduce a
topology such that \eqref{E:SINIT} becomes a \emph{compact metric space}. This
requires a suitable notion of convergence of Young measures. Recall that Young
measures are elements of the dual space $\BBE^*$ (see \eqref{E:ESTAR}), which
comes equipped with the weak* topology. By Banach-Alaoglu theorem, bounded sets
in $\BBE^*$ are weak* precompact. Moreover, since the Banach space $\BBE$ (see
\eqref{E:BBE}) is separable, the weak* topology is \emph{metrizable} on such
bounded sets, so that \emph{compactness and sequential compactness are
equivalent}. We will say that a sequence of dissipative solutions $(\epsilon^k,
\nu^k) \in S$ converges to $(\epsilon, \nu)$ if 
\begin{equation}
	\left.\begin{aligned}
		\epsilon^k & \WEAK \epsilon
\\
		\nu^k & \WEAK \nu 
	\end{aligned} \, \right\}
	\quad\text{weak* in $\BBE^*$ as $k\to\infty$,}
\label{E:YOUNC}
\end{equation}
which is defined in terms of testing against functions in $\BBE$.

%========== LEMMA
\begin{lemma}
The set $S$ is a (sequentially) compact metric space under \eqref{E:YOUNC}
\end{lemma}

%========== PROOF
\begin{proof}
It suffices to prove that every sequence of dissipative solutions $(\epsilon^k,
\nu^k) \in S$ admits a subsequence converging as in \eqref{E:YOUNC} to a pair
of Young measures $(\epsilon, \nu)$ that is again a dissipative solution in
$S$. We establish uniform boundedness of $\epsilon^k$ and $\nu^k$ in the
$\BBE^*$-norm and apply Banach-Alaoglu theorem. Let $\bar{E}, \bar{M}$ be the
total energy and second moment determined by $(\bar{\RHO}, \bar{\BM})$; see
\eqref{E:ENIN} and \eqref{E:MOIN}. Then
\[
	\|\epsilon^k(t,\cdot)\|_{\M(\DR^d\times\BFX)}
		= \int_{\DR^d\times\BFX} \epsilon^k(t,dx,d\FX)
		= \int_{\DR^d} \llbracket \RHO + \HA|\BM|^2/\RHO+U(\RHO) 
			\rrbracket_{\epsilon^k}(t,dx),
\]
which is bounded by $1+\bar{E}$ for every $t\in[0,\infty)$ uniformly in $k$,
because of Definition~\ref{D:DISSI}~\eqref{L:EPSILON}/\eqref{L:COMPATIBILITY}.
Recall that the pairing \eqref{E:YPAIR} is linear in $f$. It follows that
\[
	\sup_{k\in\N} \|\epsilon^k\|_{\BBE^*} \LS 1+\bar{E};
\]
see \eqref{E:ESTARBD}. Extracting a subsequence (not relabeled, for
simplicity), we obtain \eqref{E:YOUNC}, for some $\epsilon \in \BBE^*$. The same
argument applies to the Young measures $\nu^k$.

We can now use Lemma~\ref{L:CLOSED} to finish the proof.
\end{proof}

%========== LEMMA
\begin{lemma}\label{L:CLOSED} Suppose that a sequence of dissipative solutions
$(\epsilon^k, \nu^k) \in S$ converges to a pair of Young measures $(\epsilon,
\nu)$ in the sense of \eqref{E:YOUNC}. Then $(\epsilon, \nu) \in S$. 
\end{lemma}

%========== LEMMA
\begin{proof}
We split the proof into two steps.

\medskip

\textbf{Step 1.} We will first establish the existence of curves $t \mapsto
(\RHO_t, \BM_t)$, for which we will check the properties of
Definition~\ref{D:DISSI}. Consider densities/momenta $(\RHO^k, \BM^k)$ that are
associated with the Young measures $(\epsilon^k, \nu^k)$ as in property (6) of
Definition~\ref{D:DISSI}. We will show that for a suitable subsequence (not
relabeled) we have
\begin{equation}
	\left.\begin{aligned}
		\RHO^k_t & \WEAK \RHO_t
\\
		\BM^k_t & \WEAK \BM_t
	\end{aligned} \, \right\}
	\quad\text{weak* in the sense of measures}
\label{E:POINTW}
\end{equation}
(testing against functions in $\CB(\R^d; \R^D)$), pointwise for all
$t\in[0,\infty)$.

In order to prove \eqref{E:POINTW}, we use the refined version of the
Arzel\`{a}-Ascoli theorem in Proposition~3.3.1 of
\cite{AmbrosioGigliSavare2008}. For any $T>0$ fixed, the curves $t \mapsto
\RHO_t^k$ have second moments bounded by \eqref{E:MOMM} for all $t\in[0,T]$,
uniformly in $k$. Therefore, for each such $t$, the set $\{ \RHO^k_t
\}_{k\in\N}$ is tight, hence narrowly precompact in $\SP(\R^d)$, by Prokhorov's
theorem. This implies in particular the weak* precompactness in the sense of
measures on $\R^d$. The curves are Lipschitz continuous with respect to the
Wasserstein distance, with Lipschitz constant bounded uniformly in $k$, because
of \eqref{E:LIPWAS}. Then Proposition~3.3.1 in \cite{AmbrosioGigliSavare2008}
establishes the convergence of a subsequence (not relabeled) of the $\RHO^k$ and
a limit density $\RHO$ with \eqref{E:POINTW} for every $t\in[0,T]$. The curve $t
\mapsto \RHO_t$ is again Lipschitz continuous with respect to the Wasserstein
distance, and estimates \eqref{E:MOMM} and \eqref{E:LIPWAS} still hold.
Repeating this argument for $T \to \infty$, extracting subsequences as
necessary, we obtain \eqref{E:POINTW} for all $t\in[0,\infty)$. Notice that
narrow convergence of measures in $\SP(\R^d)$ is metrizable (see Remark~5.1.1 of
\cite{AmbrosioGigliSavare2008}) and that the Wasserstein distance $\WAS$ is
lower semicontinuous with respect to narrow convergence; see (2.1.1) in
\cite{AmbrosioGigliSavare2008}.

We will now show that the limit density $\RHO_t$ is in fact uniquely determined
by the Young measure $\nu_t$ for a.e.\ $t\in [0,\infty)$. In particular, the
limit is independent of the particular choice of subsequence used in the
Arzelà-Ascoli compactness argument, and so the whole sequence converges for
a.e.\ $t$. Indeed, since the map
\[
	\big( t,x,(\RHO,\BM) \big) \mapsto
		\eta(t) \varphi(x) \, \RHO/h(\RHO,\BM)
\]
with $\eta \in \L^1([0,\infty))$ and $\varphi \in \C(\DR^d)$ is an element in
$\BBE$, we have
\begin{multline*}
	\int_0^\infty \eta(t) \int_{\DR^d} \varphi(x) \,
			\lbrack \RHO \rbrack_\nu(t,dx) \,dt
		= \lim_{k\to\infty} \int_0^\infty \eta(t) \int_{\DR^d} \varphi(x) \,
			\lbrack \RHO \rbrack_{\nu^k}(t,dx) \,dt
\\
		= \lim_{k\to\infty} \int_0^\infty \eta(t) \int_{\R^d} \varphi(x) 
			\, \RHO^k_t(dx) \,dt
		= \int_0^\infty \eta(t) \int_{\R^d} \varphi(x) \, \RHO_t(dx) \,dt.
\end{multline*}
For the second equality we have used that \eqref{E:COMPAT} holds for $\nu^k$;
the third identity follows from \eqref{E:POINTW} and dominated convergence.
Recall that $\RHO^k_t, \RHO_t$ have finite second moments, uniformly in $k$ and
$t$ bounded. Therefore the domain of integration can be restricted to $\R^d$.
Since $\eta, \varphi$ were arbitrary, we conclude that $\lbrack
\RHO\rbrack(t,dx) = \RHO_t(dx)$ for a.e.\ $t \in [0,\infty)$. The argument works
the same for $\epsilon^k$ and $\epsilon$.

For the momentum we argue similarly: On compact time intervals $[0,T]$, the
$\BM^k_t$ are bounded in $\M_t$ (see \eqref{E:MT}) uniformly in $k$, which
implies uniformly bounded first moments and therefore tightness/narrow
precompactness. On $\M_t$ the topology of narrow convergence coincides with the
one induced by the bounded Lipschitz norm; see Corollary~3.2 and Remark~3.2 in
\cite{HilleSzarekWormZiemlanska2021}. The curves $t \mapsto \BM_t$ are Lipschitz
continuous with respect to the bounded Lipschitz norm, with Lipschitz constant
bounded by \eqref{E:LIPWAS} uniformly in $k$. Then Proposition~3.3.1 in
\cite{AmbrosioGigliSavare2008} establishes convergence of a subsequence of the
$\BM^k$ (not relabeled) towards a limit momentum that is again Lipschitz
continuous. Also $\BM_t \in \M_t$ for all $t\in[0,T]$ and estimate
\eqref{E:LIPWAS} still holds. We let $T\to \infty$ and extract further
subsequence as needed to obtain \eqref{E:POINTW}. To show that we also have
convergence when testing against functions in $\AF$, which may grow linearly at
infinity, we can adapt the truncation argument of Lemma~\ref{L:MOMS}.

We can now argue as before to show that the limit momentum $\BM_t$ is uniquely
determined by the Young measures $(\nu, \epsilon)$ for a.e.\ $t\in[0,\infty)$.
In particular, density and momentum $(\RHO,\BM)$ satisfy the compatibility
condition in Property~\eqref{L:COMPATIBILITY} of Definition~\ref{D:DISSI}. Since
all $(\RHO^k, \BM^k)$ have the same initial data $(\bar{\RHO}, \bar{\BM})$, the
convergence \eqref{E:POINTW} with $t=0$ implies that $(\RHO, \BM)$ also has
initial data $(\bar{\RHO}, \bar{\BM})$, which proves Property~\eqref{L:DATA} of
Definition~\ref{D:DISSI}. Lipschitz continuity in \eqref{L:LIPSCHITZ} has
already been discussed, and the a priori bounds \eqref{L:BOUNDS} and
\eqref{E:SPACES} can be derived from weak* precompactness and lower
semicontinuity of the internal and kinetic energies; see also
Remark~\ref{R:ENEBO}.

\medskip

\textbf{Step~2.} Property~\eqref{L:VELOCITY} of Definition~\ref{D:DISSI} follows
from lower semicontinuity of 
\begin{equation}
	(\RHO_t, \BM_t) \mapsto \sup_{\zeta\in\CB(\R^d; \R^d)}
		\int_{\R^d} \Big( \zeta(x)\cdot \BM_t(dx) 
			-\HA|\zeta(x)|^2 \,\RHO_t(dx) \Big),
\label{E:KINET}
\end{equation}
which represents the kinetic energy, under weak* convergence of measures. Notice
that the functional \eqref{E:KINET} gives $+\infty$ if $\BM_t$ is not absolutely
continuous with respect to $\RHO_t$. Then we use the weak* convergence
\eqref{E:POINTW} for every $t \in [0,\infty)$.

Property~\eqref{L:EPSILON} of Definition~\ref{D:DISSI} follows from
\eqref{E:YOUNC} and Helly's selection theorem for sequences of monotone
functions; see also Step~3 in the proof of Proposition~\ref{P:EXISTENCE}. To
prove Property~\eqref{L:NU} we observe that the weak* convergence of the Young
measures in $\BBE^*$ implies weak* convergence in $\L^\infty([0,\infty))$ of
$E^k$ and $N^k$, which are defined as in \eqref{E:TOTEN} and \eqref{E:NEN} with
$\epsilon^k, \nu^k$ in place of $\epsilon, \nu$. Indeed the map
\[
	\big( t,x,(\RHO,\BM) \big) 
		\mapsto \eta(t) \, \Big( \HA|\BM|^2/\RHO + U(\RHO) \Big) / h(\RHO,\BM)
\]
with $\eta \in \L^1([0,\infty))$ is an element in $\BBE$ (no $x$-dependence).
Choosing test functions $\eta_T := \ONE_{\{N>E\} \cap [0,T]}$, which are in
$L^1([0,\infty))$ for any $T \in \N$, we find that
\begin{align*}
	0 & \LS \int_0^\infty \eta_T(t) \big( N(t)-E(t) \big) \,dt 
\\
		& \LS \lim_{k\to \infty} \bigg( 
				\int_0^\infty \eta_T(t) \big( N(t)-N^k(t) \big) \,dt
				- \int_0^\infty \eta_T(t) \big( E(t)-E^k(t) \big) \,dt
			\bigg) = 0
\end{align*}
because $N^k(t) \LS E^k(t)$ for all $k\in\N$ and a.e.\ $t\in[0,\infty)$, and
$\eta_T \GS 0$. Thus 
\[
	N(t) \LS E(t)
	\quad\text{for a.e.\ $t\in[0,\infty)$.}
\]
Recall that the energy bounds imply the function space inclusions
\eqref{E:SPACES}.

Finally, we observe that for $\eta \in C^1_c([0,\infty))$ and $\zeta \in \AF$
\begin{align*}
	& \int_0^\infty \eta(t) \int_{\DR^d} \nabla\zeta(x) : \Big( 
		\lbrack \BM\otimes\BM/\RHO \rbrack_{\nu^k}(t,dx) 
			+ \llbracket P(\RHO) \rrbracket_{\epsilon^k}(t,dx) \Big) \,dt
\\
	& \qquad \longrightarrow
		\int_0^\infty \eta(t) \int_{\DR^d} \nabla\zeta(x) : \Big( 
			\lbrack \BM\otimes\BM/\RHO \rbrack_\nu(t,dx)
				+ \llbracket P(\RHO) \rrbracket_\epsilon(t,dx) \Big) \,dt
\end{align*}
as $k\to\infty$ because the map
\[
	\big( t, x, (\RHO, \BM) \big) \mapsto \eta(t) \nabla\zeta(x) :
		\big( \BM\otimes\BM/\RHO \big)/h(\RHO,\BM)
\]
belongs to the space $\BBE$; see \eqref{E:BBE}. Then we use \eqref{E:YOUNC}
again. The same argument works for the pressure term. On the other hand, we have
that
\begin{equation}
	\int_0^\infty \eta'(t) \int_{\DR^d} \zeta(x) \cdot \BM^k_t(dx) \,dt
		\longrightarrow \int_0^\infty \eta'(t) \int_{\DR^d} 
			\zeta(x) \cdot \BM_t(dx) \,dt
\label{E:MOPP}
\end{equation}
as $k\to\infty$. Indeed a truncation argument like the one in the proof of
Lemma~\ref{L:MOMS} and \eqref{E:POINTW} establishes convergence of the spatial
integrals in \eqref{E:MOPP} for a.e.\ $t\in[0,\infty)$; recall that the momentum
$\BM_t$ is in $\M_t$, which is defined in \eqref{E:MT}. Since the spatial
integrals are bounded uniformly in $t, k$, by Cauchy-Schwarz inequality and
\eqref{E:ENRGY}/\eqref{E:MOMM}, we can then pass to the limit in \eqref{E:MOPP}
using dominated convergence. Recall that $\eta$ has compact support in
$[0,\infty)$. This proves that $(\epsilon, \nu)$ and $\BM$ satisfy the momentum
equation in \eqref{E:DISSI} in duality with $\C^1_c([0,\infty)) \otimes \AF$.
The continuity equation can be handled in a similar way, which proves
Property~\eqref{L:CONSLAWS} of Definition~\ref{D:DISSI}. 

Finally, to prove Property~\eqref{L:DECOMPOSITION} we argue as in Step~4 of the
proof of Proposition~\ref{P:EXISTENCE}: We introduce weakly measurable maps
$\upsilon^k \colon [0,\infty)\times \R^d \longrightarrow \SP(X_V)$ with
\[
	\int_0^\infty \int_{\R^d} \phi(t,x) \int_{X_V} \varphi(\FX) 
			\,\upsilon^k_{t,x}(d\FX) \,dx \,dt 
		:= \int_0^\infty \int_{\R^d} \phi(t,x) \varphi\big(
			r^k_t(x), m^k_t(x) \big) \,dx \,dt
\]
for all $\phi \in \C_c([0,\infty)\times\R^d)$ and $\varphi \in \C_c(X_V)$. Here
$(r^k, m^k)$ is the Lebesgue density of density/momentum $(\RHO^k, \BM^k)$. As
before, one can show that the sequence of $\upsilon^k$ admits a weak* converging
subsequence (not relabeled) such that
\begin{align*}
	& \int_0^\infty \int_{\R^d} \phi(t,x) \varphi\big(
		r^k_t(x), m^k_t(x) \big) \,dx \,dt
\\
	& \qquad
		\longrightarrow \int_0^\infty \int_{\R^d} \phi(t,x) 
			\int_{X_V} \varphi(\FX) \,\upsilon_{t,x}(d\FX) \,dx \,dt 
	\quad\text{as $k\to\infty$}
\end{align*}
for all $\phi \in \C_c([0,\infty)\times\R^d)$ and $\varphi\in\C_c(X_V)$, where
$\upsilon \colon [0,\infty)\times \R^d \longrightarrow \SP(X_V)$ is some weakly
measurable map. This limit $\upsilon$ is uniquely determined by identity
\eqref{E:EQUL}, which proves in particular that the whole sequence converges.
Again we have \eqref{E:IDES} and \eqref{E:FIRS}, by Jensen's inequality.
Combining this with \eqref{E:EQLL}, we get \eqref{E:LSS} in the sense of
symmetric matrices. The difference between the left- and right-hand sides of
\eqref{E:LSS} define the defect measures $\DSC, \DSP$ in
Definition~\ref{D:DISSI}~\eqref{L:DECOMPOSITION}. For the second Young measure
$\epsilon$ and $\llbracket \cdot \rrbracket$ the argument works the same.
\end{proof}

For $(\epsilon, \nu) \in S$ we define the set of predecessors
\[
	P(\epsilon, \nu) := \big\{ (\tilde\epsilon, \tilde\nu) \in S \colon
		(\tilde\epsilon, \tilde\nu) \PREC (\epsilon, \nu) \big\},
\]
where $\PREC$ is the quasi-order introduced in Definition~\ref{D:QO}.

%========== LEMMA
\begin{lemma}
For every $(\epsilon, \nu) \in S$ the set $P(\epsilon, \nu)$ is closed under
\eqref{E:YOUNC}.
\end{lemma}

%========== PROOF
\begin{proof}
Because of Lemma~\ref{L:CLOSED}, we know that if $(\epsilon^k, \nu^k) \in S$
converges to $(\tilde\epsilon, \tilde\nu)$ in the sense of \eqref{E:YOUNC}, then
$(\tilde\epsilon, \tilde\nu) \in S$. Suppose that $(\epsilon^k, \nu^k) \PREC
(\epsilon, \nu)$, i.e., that
\begin{equation}
	a(t|\epsilon^k, \nu^k) \LS a(t|\epsilon, \nu)
	\quad\text{for a.e.\ $t\in [0,\infty)$}
\label{E:INEQQ}
\end{equation}
and all $k\in\N$. We argue as in the proof of Lemma~\ref{L:CLOSED}: Since the
maps
\[
	\big( t,x, (\RHO,\BM) \big) \mapsto
		\eta(t) \, \begin{Bmatrix} \BM\otimes\BM/\RHO \\ P(\RHO) \end{Bmatrix} / h(\RHO, \BM) 
\]
with $\eta \in \L^1([0,\infty))$ belong to $\BBE$ (no $x$-dependence),
\eqref{E:YOUNC} implies that
\begin{equation}
	a(\cdot|\epsilon^k, \nu^k) \longrightarrow a(\cdot| \tilde\epsilon, \tilde\nu)
	\quad\text{weak* in $\L^\infty\big( [0,\infty) \big)$.}
\label{E:WEAKST}
\end{equation}
Choosing test functions $\eta_T := \ONE_{\{a(\cdot|\tilde\epsilon, \tilde\nu) >
a(\cdot|\epsilon, \nu)\} \cap [0,T]}$, which are in $\L^1([0,\infty))$ for any
$T\in \N$, we use inequality \eqref{E:INEQQ} and \eqref{E:WEAKST} to obtain the
estimate
\begin{align*}
	0 & \LS \int_0^\infty \eta_T(t) \big( a(t|\tilde\epsilon, \tilde\nu)
		- a(t|\epsilon, \nu) \big) \,dt 
\\
		& \LS \lim_{k\to \infty}
				\int_0^\infty \eta_T(t) \big( a(t|\tilde\epsilon, \tilde\nu) 
					- a(t|\epsilon^k, \nu^k) \big) \,dt
			= 0.
\end{align*}
It follows that $a(t|\tilde\epsilon, \tilde\nu) \LS a(t|\epsilon, \nu)$ for
a.e.\ $t\in[0,\infty)$, and thus $(\tilde\epsilon, \tilde\nu) \in P(\epsilon,
\nu)$.
\end{proof}

Applying Theorem~\ref{T:WARD}, we then obtain our main result
Theorem~\ref{T:SELECTION}.

%%%%%%%%%%%%%%%%%%%%%%%%%%%%%%%%%%%%%%%%%%%%%%%%%%%%%%%%%%%%%%%%%%%%%%%%%%%%%%%%
%%%%%%%%%%%%%%%%%%%%%%%%%%%%%%%%%%%%%%%%%%%%%%%%%%%%%%%%%%%%%%%%%%%%%%%%%%%%%%%%
%%%%%%%%%%%%%%%%%%%%%%%%%%%%%%%%%%%%%%%%%%%%%%%%%%%%%%%%%%%%%%%%%%%%%%%%%%%%%%%%
%%%%%%%%%%%%%%%%%%%%%%%%%%%%%%%%%%%%%%%%%%%%%%%%%%%%%%%%%%%%%%%%%%%%%%%%%%%%%%%%
%%%%%%%%%%%%%%%%%%%%%%%%%%%%%%%%%%%%%%%%%%%%%%%%%%%%%%%%%%%%%%%%%%%%%%%%%%%%%%%%

\section*{Data availability statement}

Data sharing not applicable to this article as no datasets were generated or
analysed during the current study.

\printbibliography

% bibtex
% \bibliography{Selection}{}
% \bibliographystyle{alpha}

\end{document}